\documentclass{scrartcl}

\usepackage[english]{babel}
\usepackage[utf8]{inputenc}
\usepackage{mathtools}
\usepackage{mathrsfs}
\usepackage{amsmath}
\usepackage{amssymb}
\usepackage{amsthm}
\usepackage{mathpazo}
\usepackage{graphicx}
\usepackage[caption=false]{subfig}
\usepackage[top=1.25in]{geometry}
\usepackage{xspace}
\usepackage{ifthen}
\usepackage{enumitem}
\setenumerate{label=\textup{(\alph*)}}

\usepackage[hyphens]{url}
\usepackage[hidelinks]{hyperref}
\usepackage{cleveref}

\makeatletter
\DeclareOldFontCommand{\sc}{\normalfont\scshape}{\@nomath\sc}
\makeatother

\newtheorem{definition}{Definition}
\numberwithin{definition}{section}
\newtheorem{theorem}[definition]{Theorem}
\newtheorem{lemma}[definition]{Lemma}

\newtheorem{proposition}[definition]{Proposition}
\newtheorem{assumption}[definition]{Assumption}

\theoremstyle{definition}

\crefname{customlem}{Lemma}{Lemmas}
\Crefname{customlem}{Lemma}{Lemmas}
\crefname{customthm}{Theorem}{Theorems}
\Crefname{customthm}{Theorem}{Theorems}

\Crefrangeformat{enumi}{Assumptions \normalfont{#3#1#4--#5#2#6}}
\crefname{corollary}{Corollary}{Corollaries}
\Crefname{assumption}{Assumption}{Assumptions}
\Crefrangeformat{assumption}{Assumptions \normalfont{#3#1#4--#5#2#6}}
\crefname{theorem}{Theorem}{Theorems}
\crefname{lemma}{Lemma}{Lemmas}
\crefname{proposition}{Proposition}{Propositions}

\numberwithin{equation}{section}
\crefname{enumi}{Assumption}{Assumptions}

\newlist{enumthm}{enumerate}{1}
\setlist[enumthm]{label=\textup{(\alph*)},ref=\theassumption~\textup{(\alph*)}}
\crefalias{enumthmi}{assumption}

\newcommand{\dualpHzeroone}[3][]
{\langle #2, #3 \rangle_{H^{-1}(#1), H_0^1(#1)}}

\newcommand{\norm}[2][2]{\|#2\|_{#1}}
\newcommand{\eqeqnorm}[2][2]{|#2|_{#1}}

\DeclarePairedDelimiterXPP\cnorm[2]{}\lVert\rVert{_{#1}}{#2}
\newcommand{\functionArgument}[1]{\ifthenelse{\equal{#1}{}}  %
	{}
	{({#1})}
}

\newcommand{\youngfun}[1]{\psi\functionArgument{#1}}
\newcommand{\luxemburg}[3][B]{\norm[L_{\youngfun{}}(\Omega; #1)]{#3}}
\newcommand{\dualp}[3][]{\langle #2, #3 \rangle_{#1^*, #1}}
\newcommand{\dualpb}[3][]{\langle #2, #3 \rangle_{#1}}
\newcommand{\inner}[3][]{( #2, #3 )_{#1}}
\DeclarePairedDelimiterX\innerp[3]{(}{)_{#1}}{#2,#3}
\newcommand{\rdc}{\kappa}
\newcommand{\quasiinter}{\mathcal{I}_h}
\newcommand{\pobj}{J}
\newcommand{\rpobj}{\widehat{\pobj}}

\newcommand{\erpobj}{F}
\DeclareSymbolFont{ugrf@m}{U}{eur}{m}{n}
\DeclareMathSymbol{\upkappa}{\mathord}{ugrf@m}{"14}
\newcommand{\rrdc}{\upkappa}

\renewcommand{\natural}{\mathbb{N}}
\newcommand{\csp}{U}
\newcommand{\adcsp}{\csp_\text{ad}}
\newcommand{\ssp}{Y}
\newcommand{\ssph}{\ssp_h}
\newcommand{\csph}{\csp_h}
\newcommand{\hsp}{H}
\newcommand{\lb}{\mathfrak{l}}
\newcommand{\ub}{\mathfrak{u}}
\newcommand{\domain}{D}

\newcommand{\real}{\mathbb{R}}
\newcommand{\sblf}[2][\real]{\mathscr{L}(#2, #1)}
\newcommand{\adcsph}{\csp_{\text{ad}, h}}
\newcommand{\embedding}{\xhookrightarrow{}}

\newcommand{\tfa}{\text{for all}}

\newcommand{\gateaux}{G\^ateaux}
\newcommand{\Caratheodory}{Carath\'eodory}
\newcommand{\friedrichs}{Friedrichs\xspace}
\newcommand{\cea}{C{\'e}a\xspace}
\newcommand{\tand}{\text{and}}

\newcommand{\eu}{\ensuremath{\mathrm{e}}}
\newcommand{\du}{\ensuremath{\mathrm{d}}}

\newcommand{\cF}{\mathcal{F}}

\DeclarePairedDelimiterXPP\cE[1]{\mathbb{E}}[]{}{%
	
	#1}
\DeclarePairedDelimiterXPP\Prob[1]{\mathrm{Prob}}(){}{
	
	#1}

\newcommand{\wlb}{a}
\newcommand{\wub}{b}

\newcommand{\prox}[3][]{\mathrm{prox}_{#2}\functionArgument{#3}}

\newcommand{\uhN}{u_{h, N}^*}
\newcommand{\rv}{Z}
\newcommand{\erpobjhN}{\hat{\erpobj}_{h,N}}
\newcommand{\spL}[2]{\mathscr{L}(#1, #2)}

\newcounter{param}
\newcommand{\newconstant}{%
	\refstepcounter{param}%
	\ensuremath{C_\theparam}}
\newcommand{\oldconstant}[1]{\ensuremath{C_{\ref{#1}}}}

\DeclareMathOperator*{\argmin}{arg\,min}

\newcommand{\keywords}[1]
{
	{\small	
		\textbf{Key words.} {#1}
		\\
	}
}

\newcommand{\amssubject}[1]
{
	{\small
		\textbf{AMS subject classifications.} {#1}
	}
}

\renewcommand{\abstract}[1]
{
	{\small
		\textbf{Abstract.} {#1}
		\\
	}
}

\title{Reliable Error Estimates for Optimal Control of 
	Linear Elliptic PDEs with Random Inputs}
\author{Johannes Milz\thanks{H.\ Milton Stewart School of Industrial and Systems Engineering, Georgia Institute of Technology, Atlanta, Georgia 30332, 
		USA  (\texttt{johannes.milz@isye.gatech.edu}). The
project was partially supported by the Deutsche Forschungsgemeinschaft 
(DFG, German Research Foundation)---project 188264188/GRK1754.}}
\date{April 16, 2023}

\begin{document}

\maketitle

\abstract{%
	We discretize a risk-neutral optimal control problem
	governed by a linear elliptic partial differential equation
	with random inputs using a Monte Carlo sample-based approximation and 
	a finite element
	discretization, yielding finite dimensional
	control problems. We establish an exponential tail bound
	for the distance between the finite dimensional problems'
	solutions and the risk-neutral problem's solution.
	The tail bound implies
	that solutions to the risk-neutral optimal control problem
	can be reliably estimated with the solutions to the 
	finite dimensional control
	problems. Numerical simulations illustrate our theoretical findings.
}
\par 
\par 
\amssubject{%
	90C15, 90C60,   35Q93, 35R60, 49M25,  49N10, 65M60, 65C05
}

\section{Introduction}
\label{sec:intro}

Many applications in science and engineering require the solution
of optimization problems constrained by physics-based
models, such as those modeled
using partial differential equations (PDEs).
Optimization problems with PDE constraints
include, for instance, tidal-stream array optimization  \cite{Goss2021},
optimal design of permanent magnet synchronous machines \cite{Alla2019}, and
topology design of cantilever structures \cite{Xia2021}.
When parameters in PDEs are uncertain, we obtain a 
parameter-dependent objective function. The  solutions
to the parameterized PDE-constrained optimization may depend significantly on 
the particular parameter choices. In such cases, 
we may be interested in obtaining a decision resilient to uncertainty. 
We model the parameter vector as random vector
and seek a  decision performing
best on average. The task is formulated as a risk-neutral 
PDE-constrained optimization
problem. Its objective function is defined as the expected value
of the parameter-dependent objective function. Risk-neutral PDE-constrained
optimization problems are challenging both from a theoretical 
and practical perspective. Their feasible sets 
are often infinite dimensional, high-dimensional random vectors
result in high-dimensional integrals, and PDEs commonly lack 
closed-form solutions. For numerical simulations, 
risk-neutral control problems 
must generally be discretized and expectations be approximated.  
We combine a Monte Carlo sample-based approximation of expectations
with finite dimensional approximations of the decision 
and PDE solution spaces 
to obtain a finite dimensional optimization problem, ready to be
solved using state-of-the-art solution methods for PDE-constrained optimization.
We derive error bounds on the distance between 
solutions to the finite dimensional problem
and that to the risk-neutral problem, which are 
valid with high probability. 
When formulating the risk-neutral optimal control problem,
we allow for a particular nonsmooth control regularization, as it appears 
in control placement applications such as the design of tidal-stream
renewable energy farms \cite{Funke2016}.

We consider the  optimal control problem
governed by an  elliptic PDE
\begin{align}
	\label{eq:saa:lqcp}
	\min_{u \in \adcsp} \, 
	(1/2)\cE{\norm[L^2(\domain)]{S(u, \xi)-y_d}^2} 
	+ (\alpha/2)\norm[L^2(\domain)]{u}^2
	+ \gamma \norm[L^1(\domain)]{u},
\end{align}
where $y_d \in L^2(\domain)$, $\domain \subset \real^d$
with $d \in \{1,2,3\}$
is a convex,
polygonal, bounded domain,  $\alpha > 0$ and $\gamma \geq 0$
are regularization parameters, and for
each $(u,\xi) \in L^2(\domain) \times \Xi$, 
$y_{\xi} = S(u, \xi) \in H_0^1(\domain)$ solves the
parameterized elliptic
\begin{align}
	\label{eq:S}
	\int_{\domain}\rdc(\xi)
	\nabla y_{\xi} \cdot \nabla v \, \du x = 
	\int_{\domain} uv \, \du x
	\quad \tfa \quad v \in H_0^1(\domain).
\end{align}
Here, $\xi$ is a random element mapping from a probability space
to a complete, separable metric space $\Xi$. The space $\Xi$ is equipped
with its Borel sigma-field. For the numerical simulations presented
in \cref{sec:simulations}, $\Xi$ is a closed subset of $\real^p$. 
The space $L^2(\domain)$ is the space of square integrable functions
defined on $\domain$ and $L^1(\domain)$ is that of integrable functions. 
Moreover, $H_0^1(\domain)$ is the Sobolev
space of $L^2(\domain)$-functions with
square integrable weak derivatives and zero boundary traces. 
We require the random diffusion coefficient 
$\rdc : \Xi \to C^1(\bar{\domain})$
be essentially bounded
and assume that there exist  $\rdc_{\min} > 0$
such that  $\rdc_{\min} \leq \rdc(\xi)(x)$
for all $(\xi,x) \in \Xi \times \bar{\domain}$.
The terms $\kappa(\xi)$, $\nabla y_{\xi}$, $\nabla v$, $v$, and $u$ in  \eqref{eq:S} are evaluated at $x \in \domain$, but we omit writing these
evaluations. Throughout the text, we  omit similar evaluations 
in other equations as well.
The inner product of a real Hilbert
space $\hsp$ is denoted by $\inner[\hsp]{\cdot}{\cdot}$
and its norm by $\norm[\hsp]{\cdot} = \inner[\hsp]{\cdot}{\cdot}^{1/2}$.
The feasible set $\adcsp$ is defined by
\begin{align}
	\label{eq:adcsp}
	\adcsp = \{\, u \in L^2(\domain): \, \lb \leq u \leq \ub
	\, \, \text{a.e. in} \,\, \domain  \, \},
	\quad \text{where} \quad 
	\lb \leq 0 \leq  \ub \quad  \text{and} \quad \lb,\, \ub \in \real.
\end{align}

We discretize the control problem \eqref{eq:saa:lqcp}
using the sample average approximation (SAA) \cite{Kleywegt2002,Shapiro2021}
and a finite element discretization \cite{Hinze2009,Troeltzsch2010a}.
To approximate the expected value in \eqref{eq:saa:lqcp}, 
let $\xi^1$, $\xi^2$, $\ldots$ be 
defined on 
a complete probability space $(\Omega, \cF, P)$
and be 
independent identically distributed
$\Xi$-valued random elements,  
each having the same distribution as 
that of $\xi$. 
For a discretization parameter $h \in (0,1)$, 
let $\ssph$ and $\csph$ be finite dimensional subspaces, such
as finite element spaces, of
the state space $H_0^1(\domain)$ and the control
space $L^2(\domain)$, respectively.

Now, we can formulate the discretized SAA problem
\begin{align}
	\label{eq:saa:femsaa}
	\min_{u_h\in\adcsph}\, 
	\frac{1}{2N}
	\sum_{i=1}^N\norm[L^2(\domain)]{S_h(u_h, \xi^i)-y_d}^2
	+(\alpha/2)\norm[L^2(\domain)]{u_h}^2
	+ \gamma \norm[L^1(\domain)]{u_h},
\end{align}
where  $N \in \natural$ is the sample size, 
$\adcsph = \adcsp \cap \csph$ and
for each $(u,\xi) \in L^2(\domain) \times \Xi$, 
$y_{{\xi,h}} = S_h(u, \xi) \in \ssph$ solves the
discretized PDE
\begin{align}
	\label{eq:Sh}
	\int_{\domain}\rdc(\xi)\nabla y_{\xi,h} \cdot \nabla v_h  \, \du x = 
	\int_{\domain} uv_h  \, \du x
	\quad \tfa \quad v_h \in \ssph.
\end{align}
Whereas the optimal control problem \eqref{eq:saa:lqcp} is infinite dimensional
and its objective function involves a potentially high-dimensional integral,
the discretized SAA problem \eqref{eq:saa:femsaa} is a finite dimensional
optimization problem  with the expectation
approximated by the sample average. 
The discretized SAA problem can be efficiently solved using
semismooth Newton methods \cite{Mannel2020,Stadler2009,Ulbrich2011}.

We analyze the accuracy and reliability
of solutions to the discretized SAA problem \eqref{eq:saa:femsaa}
as approximate solutions to the infinite dimensional PDE-constrained
problem \eqref{eq:saa:lqcp}.
The manuscript's main novelty is the derivation 
of an error estimate allowing us to conclude that
(depending on the discretization 
parameter $h$ and sample size $N$), the finite dimensional
sample-based solutions are close to the 
infinite dimensional problem's solution with 
extremely high probability.
Let $u^*$ be the solution to \eqref{eq:saa:lqcp} and
$\uhN$ be that to \eqref{eq:saa:femsaa}.
We establish the  exponential tail bound, our main contribution, 
\begin{align}
	\label{eq:2021-02-22T15:31:07.382}
	\Prob{\norm[L^2(\domain)]{\uhN-u^*} \geq c_1 h + c_2 \varepsilon} \leq 
	2\exp(-\varepsilon^2N/2) 
	\quad \tfa \quad \varepsilon > 0,
\end{align}
where $c_1$, $c_2 \in (0, \infty)$ are deterministic, 
problem-dependent parameters, 
which we explicitly derive within our error estimation.
The constants $c_1$ and $c_2$ depend, for example, on the
regularization parameter $\alpha > 0$ and deterministic characteristics of 
the random diffusion coefficient.
The exponential tail bound \eqref{eq:2021-02-22T15:31:07.382} 
implies that for each $\delta \in (0,1)$, 
with a probability of at least $1-\delta$,
\begin{align}
	\label{eq:2021-02-22T16:35:32.14}
	\norm[L^2(\domain)]{\uhN-u^*} 
	< c_1 h +c_2 \sqrt{\frac{2\ln(2/\delta)}{N}}.
\end{align}
This estimate essentially expresses the fact
that the random vector $\uhN$ is ``close''
to the solution $u^*$ to \eqref{eq:saa:lqcp}
with high probability, provided that $0 < \delta \ll 1$.
Since the inverse of $\delta \in (0,1)$ only appears logarithmically in 
\eqref{eq:2021-02-22T16:35:32.14}, we can choose a small 
``error probability'' $\delta$, 
say $\delta = 10^{-12}$. Therefore, we may conclude that
the  solution to the infinite dimensional
\eqref{eq:saa:lqcp} can be reliably  estimated via solutions to 
the discretized SAA problem
\eqref{eq:saa:femsaa}. We refer to the exponential tail bound
\eqref{eq:2021-02-22T15:31:07.382} as reliable error estimate.

Establishing an exponential
tail bound, such as \eqref{eq:2021-02-22T15:31:07.382}, 
is one approach to analyzing the accuracy of the discretized
SAA problem's solutions. Another approach
is to establish an expectation bound, 
providing an estimate on the expected distance between the 
SAA solutions and the true one. 
Expectation bounds can  be combined
with Tschebyshev's inequality to obtain tail bounds.
For example, suppose that there exists positive constants
$\tilde c_1$ and $\tilde c_2$ 
such that for all $h \in (0,1)$ and $N \in \natural$,
we have the expectation bound
\begin{align}
	\label{eq:exp_bound}
	\cE{\norm[L^2(\domain)]{\uhN-u^*}^2}
	\leq 
	\tilde c_1 h^2 + \tilde c_2/N.
\end{align}
Combining \eqref{eq:exp_bound} 
with Tschebyshev's inequality, we obtain the polynomial tail bound
\begin{align}
	\label{eq:tailprob}
	\Prob{\norm[L^2(\domain)]{\uhN-u^*} \geq \varepsilon} \leq 
	\varepsilon^{-2} \big(\tilde c_1 h^2 + \tilde c_2/N \big)
	\quad \tfa \quad \varepsilon > 0.
\end{align}
Comparing this tail bound with that in 
\eqref{eq:2021-02-22T15:31:07.382}, we obtain a nonexponential
decay of the tail probability in \eqref{eq:tailprob} 
as a function of the sample size $N$. Moreover, 
the discretization parameter $h$ appears on the right-hand side in 
\eqref{eq:tailprob}. 
As a result, the expectation bound \eqref{eq:exp_bound} 
may not imply an exponential tail bound
such as that in \eqref{eq:2021-02-22T15:31:07.382}. 
On the other hand, the tail bound \eqref{eq:2021-02-22T15:31:07.382}  implies
bounds on all finite moments of $\norm[L^2(\domain)]{\uhN-u^*}$ (see
\eqref{eq:new_expecatation_bound}).
In particular, the exponential tail bound \eqref{eq:2021-02-22T15:31:07.382} 
implies the expectation bound 
\begin{align}
	\label{eq:expecatation_bound}
	\cE{\norm[L^2(\domain)]{\uhN-u^*}}
	\leq 
	c_1 h + c_2\sqrt{2\pi}/\sqrt{N};
\end{align}
see \cref{sec:expectation_bound}. 
Whereas \eqref{eq:2021-02-22T16:35:32.14} expresses the fact
that $\uhN$ is ``close'' to $u^*$ with a probability of at least
$1-\delta$,  the expectation bounds in
\eqref{eq:exp_bound} and \eqref{eq:expecatation_bound}
imply that 
$\uhN$ is ``close'' to $u^*$ on average. 
The validity of the expectation bound 
\eqref{eq:exp_bound} may require fewer assumptions on the
random diffusion coefficient than those needed to derive
\eqref{eq:2021-02-22T15:31:07.382}. 
In particular, while we require
$\kappa \colon \Xi \to C^1(\bar{\domain})$ be essentially bounded
and the existence of a constant 
$\kappa_{\min} > 0$ such that $\kappa_{\min} \leq \kappa(\xi)(x)$
for all $(\xi, x) \in \Xi \times \bar{\domain}$,
the expectation bound \eqref{eq:exp_bound} may be established
for more general random diffusion coefficients, such as lognormal
random diffusion coefficients.
Moreover,  the constants $\tilde c_1$ and $\tilde c_2$ 
in \eqref{eq:exp_bound} may be smaller than $c_1$ and $c_2$ in 
\eqref{eq:2021-02-22T15:31:07.382} even if the exponential tail bound
\eqref{eq:2021-02-22T15:31:07.382} holds true.
We further comment on our choice of random diffusion coefficients
in \cref{subsec:exponentialtailbound}.

The SAA problem \eqref{eq:saa:femsaa} is a canonical approximation
of the control problem \eqref{eq:saa:lqcp} and is obtained via
two types of approximations:  finite dimensional discretizations
of the control and state spaces, and random approximations of 
expectations using sample averages. To analyze these approximations,
we make use of error analyses developed in the literature: our
derivation of the exponential tail bound 
\eqref{eq:2021-02-22T15:31:07.382} 
is inspired by the error analyses for deterministic control problems 
developed in \cite{Arnautu1998,Casas2012a,Meidner2008,Wachsmuth2011}, 
makes use of stability and error estimates derived in
\cite{Ali2017,Charrier2012,Martin2021,Martin2021a},
and uses an exponential tail bound (a large deviation-type bound) for 
$L^2(\domain)$-valued random vectors established in 
\cite{Pinelis1992,Pinelis1994}.
Large deviation-type results and tools are well-established in the
literature on stochastic programming 
\cite{Ghadimi2012,Nemirovski2009,Tong2022}
and  on uncertainty quantification
with differential equations  \cite{Dematteis2019,Tong2020}.

A comprehensive analysis of the SAA approach with a focus
on finite dimensional problems is provided in
\cite[Chap.\ 5]{Shapiro2021}. The current text is partly based on
results established in the author's dissertation \cite{Milz2021a}.
The SAA method is analyzed in \cite{Roemisch2021,Milz2021}
as applied to infinite dimensional 
risk-neutral optimal control of elliptic PDEs with random inputs, but
without considering state and control discretizations. 
Expectation bounds for  unconstrained risk-neutral elliptic control problems 
with and without state and control discretization 
and $\gamma = 0$ are derived in \cite{Martin2021} (see also \cite{Martin2021a}).
Besides deriving qualitative and quantitative stability for 
infinite dimensional risk-neutral optimization problems, 
Monte Carlo approximations and finite dimensional discretizations of 
risk-neutral elliptic control problems with $\gamma = 0$ are studied in
\cite{Hoffhues2020} using probability metrics.
Our error estimation approach differs from that in 
\cite[sect.\ 6]{Hoffhues2020} in that, for example, we estimate
$\norm[L^2(\domain)]{\uhN-u^*}$ without using the error
decomposition
$\norm[L^2(\domain)]{\uhN-u^*} \leq \norm[L^2(\domain)]{\uhN-u_N^*}
+ \norm[L^2(\domain)]{u_N^*-u^*}
$,
where $u_N^*$ is the solution to an infinite dimensional SAA problem.
Monte Carlo sampling provides one approach to approximating
the expected value in risk-neutral PDE-constrained optimization problems. 
Alternative approximation approaches are, for example,
quasi-Monte Carlo sampling \cite{Guth2021,Guth2019},
stochastic collocation and sparse grids 
\cite{Kouri2014a,Tiesler2012}, and tensor-based methods
\cite{Garreis2017,Garreis2019a}.
Risk-averse control of elliptic equations 
with uncertain fractional exponents is considered in \cite{Antil2021} 
and error estimates for
sample-based and finite element approximations are derived.
Solution methods for PDE-constrained 
optimization under uncertainty include, for example, stochastic 
gradient methods \cite{Geiersbach2019a,Martin2021,Martin2021a}
and inexact trust-region methods \cite{Garreis2019a,Kouri2014}.

The rest of the manuscript is organized as follows.
We provide further notation in 
\cref{sec:notation}
and formulate assumptions on the control problem \eqref{eq:saa:lqcp} 
in 
\cref{sec:assumptions}. In \cref{sec:assumptions} we also introduce
the state and control discretization of \eqref{eq:saa:lqcp} as a set of
two assumptions, which allows us to avoid to formally
define finite element spaces. 
The reliable error estimate is stated in
\cref{sec:error} and established in 
\cref{sec:2021-02-22T19:08:49.94}. 
We recast the PDEs in \eqref{eq:S}
and \eqref{eq:Sh} as linear systems in \cref{sec:2021-02-22T19:08:49.94}.
To establish the error estimate, we demonstrate that the solution to
\eqref{eq:saa:lqcp}  has square integrable weak derivatives, allowing
us to approximate it using a quasi-interpolation operator.
We present numerical illustrations in \cref{sec:simulations}.
The manuscript is concluded with \cref{sec:discussion}.

\section{Preliminaries and further notation}
\label{sec:notation}

If not specified otherwise, relations between random elements
are supposed to hold with probability one. Metric spaces are 
equipped with their Borel sigma-field. 
Let $\Lambda$ be a real Banach space. 
The norm of $\Lambda$ is denoted by
$\norm[\Lambda]{\cdot}$. Let
$\Lambda_1$ and $\Lambda_2$ be real Banach spaces. The space of linear, bounded
operators mapping from $\Lambda_1$ to $\Lambda_2$ is denoted by
$\sblf[\Lambda_2]{\Lambda_1}$. 
We define   $\Lambda^* = \sblf[\real]{\Lambda}$. The dual pairing
between $\Lambda^*$ and $\Lambda$ is denoted by 
$\dualp[\Lambda]{\cdot}{\cdot}$.
If $\Lambda$ is a real, reflexive Banach space, then
we identify $(\Lambda^*)^*$ with $\Lambda$
and write $(\Lambda^*)^* = \Lambda$.
The adjoint operator
of $\Upsilon \in \sblf[\Lambda_2]{\Lambda_1}$ is denoted by
$\Upsilon^* \in \sblf[\Lambda_1^*]{\Lambda_2^*}$.
Let $(\Theta, \mathcal{A}, \mu)$
be a probability space.  
An operator-valued mapping
$\Upsilon : \Theta \to \sblf[\Lambda_2]{\Lambda_1}$ is called uniformly
measurable if there exist a sequence of simple mappings
$\Upsilon_k : \Theta \to   \sblf[\Lambda_2]{\Lambda_1}$ with
$\Upsilon_k(\theta) \to \Upsilon(\theta)$ in $\sblf[\Lambda_2]{\Lambda_1}$ as 
$k \to \infty$ for all $\theta \in \Theta$.
Let $\domain \subset \real^d$ be a bounded domain. 
We define the Sobolev spaces $H^p(\domain)$
with $p \in \{1,2\}$ as the spaces of $L^2(\domain)$-functions with
square integrable weak derivatives up to order $p$
and $L^\infty(\domain)$ as the space of essentially bounded functions
defined on $\domain$. Furthermore, we define the
seminorm $\eqeqnorm[H^1(\domain)]{y} = \norm[L^2(\domain)^d]{\nabla y}$
for $y \in H^1(\domain)$, where $\nabla y$ is the weak gradient of $y$.
The norm of $H^1(\domain)$ is given by
$\norm[H^1(\domain)]{y} = 
(\norm[L^2(\domain)]{y}^2 + \eqeqnorm[H^1(\domain)]{y}^2)^{1/2}$.
We define $H^{-1}(\domain) = H_0^1(\domain)^*$. We identify
$L^2(\domain)$ with its dual and write $L^2(\domain) = L^2(\domain)^*$.
\friedrichs' constant $C_\domain > 0$ of  
the domain $\domain \subset \real^d$ is defined by
$C_\domain = \sup_{v\in H_0^1(\domain)\setminus\{0\}}\, 
\norm[L^2(\domain)]{v}/\eqeqnorm[H^1(\domain)]{v}$.
Since $\domain$ is bounded, $C_\domain < \infty$ \cite[Thm.\ 1.13]{Hinze2009}.
Let $\hsp$ be a real Hilbert space.
For a convex, lower semicontinuous, 
proper function $\chi : \hsp \to (-\infty,\infty]$, 
the proximity operator 
$\prox{\chi}{}:\hsp \to \hsp$ of $\chi$ is defined by
\begin{align}
	\label{eq:prox}
	\prox{\chi}{v}
	= \argmin_{w\in \hsp}\, 
	\chi(w) + (1/2)\norm[\hsp]{v-w}^2;
\end{align}
see \cite[Def.\ 12.23]{Bauschke2011}. 
We define the indicator function $I_{\hsp_0} : \hsp  \to [0,\infty]$ of 
$\hsp_0 \subset \hsp$ 
by $I_{\hsp_0}(v) = 0$ if $v \in \hsp_0$ and $I_{\hsp_0}(v) = \infty$
otherwise. 

\section{Assumptions on the optimal control problem}
\label{sec:assumptions}

We state assumptions on the domain, the random diffusion coefficient, 
and the state and control space discretization. 

\subsection{Domain and random diffusion coefficient}
\label{subsec:domainkappa}
We impose conditions on the domain $\domain$
and the random diffusion coefficient $\rdc$.

\begin{assumption}
	\label{ass:domainkappa}
	\begin{enumthm}[wide,nosep,leftmargin=*]
		\item 
		\label{ass:domainkappa_1}
		The domain $\domain \subset \real^d$
		with $d \in \{1,2, 3\}$ is bounded, 	convex and polyhedral. 
		\item 	
		\label{ass:domainkappa_2}
		The random diffusion coefficient
		$\rdc : \Xi \to  C^1(\bar{\domain})$
		is (strongly) measurable
		and there exists $\rdc_{\min}$, $\rdc_{\max} \in (0,\infty)$
		with $\rdc_{\min} \leq \rdc(\xi)(x) \leq \rdc_{\max}$
		for all $(\xi,x) \in \Xi \times \bar{\domain}$
		and $\rdc_{\max,1} \in (0,\infty)$ with
		$\norm[C^1(\bar{\domain})]{\rdc(\xi)} \leq \rdc_{\max,1}$
		for all $\xi \in \Xi$.
	\end{enumthm}
\end{assumption}

\Cref{ass:domainkappa} allows us
to establish higher regularity of the PDE solution
using results established in \cite{Ali2017}. 
While it may be possible to derive reliable error estimates 
in case $\rdc : \Xi \to  C^t(\bar{\domain})$
for some $t \in (0,1]$, we assume that 
\Cref{ass:domainkappa_2} holds true.
\Cref{ass:domainkappa_2}
is violated if $\kappa$ is a log-normal 
random diffusion coefficient
\cite{Charrier2012}. 

\subsection{Discretization of state and control space}
\label{subsec:saa:statecontroldiscretization}
We introduce the discretization for
the state space $H_0^1(\domain)$ 
and for the control space 
$L^2(\domain)$.

\begin{assumption}
	\label{ass:2020-11-08T15:31:39.527}
	For each $h \in (0,1)$,
	$\ssph$ is a finite dimensional subspace  of $H_0^1(\domain)$.
	For a constant $C_{\ssp} > 0$
	independent of $h \in (0,1)$, 
	\begin{align}
		\inf_{v_h \in \ssp_h} \, \eqeqnorm[H^1(\domain)]{v-v_h}
		\leq C_{\ssp}h\norm[H^2(\domain)]{v} 
		\;\; \tfa \;\; v \in H_0^1(\domain) \cap H^2(\domain)
		\;\; \tand  
		\;\; h \in (0,1).
	\end{align} 
\end{assumption}

Let \Cref{ass:2020-11-08T15:31:39.527} hold.
Since $\ssph$ is a finite dimensional subspace of $H_0^1(\domain)$,
$\ssph$ is closed \cite[Thm.\ 3.2-4]{Kreyszig1978}.
Hence $\ssph$ is a Hilbert space. 
\Cref{ass:2020-11-08T15:31:39.527} is satisfied
if \Cref{ass:domainkappa_1} holds true 
and $\ssph$ is the  space of 
piecewise linear finite elements defined on 
a certain regular meshes of $\bar{\domain}$
with zero boundary conditions (cf.\ \cite[Lem.\ 4.3]{Ali2017}).
The following assumption is based on 
\cite[Assumption~4.2]{Wachsmuth2011} and
\cite[Assumption~3.3]{Reyes2008}.
\begin{assumption}
	\label{ass:csph}
	For each $h \in (0,1)$, there exists 
	$n_h \in \natural$
	and $\phi_h^j \in L^\infty(\domain)$ 
	with  $\phi_h^j \geq 0$ a.e.\ in $\domain$,
	$\norm[L^\infty(\domain)]{\phi_{h}^j} = 1$
	for $j = 1, \ldots, n_h$, and
	$
	\sum_{j=1}^{n_h} \phi_h^j = 1
	$
	a.e.\ in $\domain$.
\end{assumption}

Let \Cref{ass:domainkappa_1,ass:csph}  hold.
For each $h \in (0,1)$, 
we define $\csph$ as the linear span
of  $\{\phi_h^j: j=1, \ldots, n_h\}$.
Since $\phi_h^j \in L^\infty(\domain)$
and $\domain$ is bounded, 
we have $\phi_h^j \in L^2(\domain)$.
Hence $\csph $ is a subspace of
$L^2(\domain)$ \cite[p.\ 56]{Kreyszig1978}.
Since $\csph$ is finite dimensional, it is
complete \cite[Thm.\ 3.2-4]{Kreyszig1978}.

When \Cref{ass:domainkappa_1} 
is fulfilled, $d \in \{2,3\}$, and $\csph$
is given by either piecewise constant  
or piecewise linear finite elements defined on regular meshes 
of $\bar{\domain}$ defined by
triangles if $d=2$ and tetrahedra if $d=3$, 
then \Cref{ass:csph} is fulfilled
\cite[Rem.\ 3.1]{Reyes2008}.
Moreover, if \Cref{ass:domainkappa_1}  holds, 
$d=1$, and $\csph$ is given by piecewise constant finite elements defined on 
intervals, then \Cref{ass:csph} holds true.

Let \Cref{ass:csph} be satisfied and let $h \in (0,1)$.
Following \cite[Def.\ 2.2]{Carstensen1999}, 
\cite[eqns.\ (10) and (11)]{Reyes2008}
and \cite[p.\ 868]{Wachsmuth2011}, 
let us define 
$\pi_h^j : L^1(\domain) \to \real$
and the quasi-interpolation operator 
$\quasiinter : L^1(\domain) \to \csph$ by 
\begin{align}
	\label{eq:saa:quasiinter}
	\pi_h^j[u] = 
	\innerp{L^2(\domain)}{\phi_{h}^j}{u}/\innerp{L^2(\domain)}{\phi_{h}^j}{1}
	\quad\quad \tand 
	\quad\quad  
	\quasiinter u = \sum_{j=1}^{n_h} \pi_h^j[u] \phi_{h}^j.
\end{align}

\Cref{ass:csph} and H\"older's inequality 
imply that $\pi_h^j$ and $\quasiinter$ and well-defined.

\begin{assumption}
	\label{ass:2020-01-31T16:12:16.968}
	For a constant $C_{\csp} > 0$
	independent of $h \in (0,1)$, we have
	for all $h \in (0,1)$ and every $u \in H^1(\domain)$, 
	\begin{align}
		\label{eq:saa:ReyesMeyerVexlerLem43}
		\norm[L^2(\domain)]{u - \quasiinter u} \leq C_{\csp} h
		\eqeqnorm[H^1(\domain)]{u} 
		\quad \tand \quad 
		\norm[H^1(\domain)^*]{u - \quasiinter u} \leq C_{\csp} h^2
		\norm[H^1(\domain)]{u}.
	\end{align}
\end{assumption}

According to \cite[Lems.\ 4.3 and 4.4]{Reyes2008},
\Cref{ass:2020-01-31T16:12:16.968} is fulfilled
if \cite[Assumption 3.3]{Reyes2008}
is satisfied. If  
\Cref{ass:domainkappa_1}
holds and $d \in \{2,3\}$, then  \cite[Rem.\ 3.1]{Reyes2008} ensures that
\cite[Assumption 3.3]{Reyes2008} is fulfilled
if $\csph$ is defined by piecewise constant 
finite elements defined on regular meshes 
of $\bar{\domain}$ defined by 
triangles if $d=2$ and tetrahedra if $d=3$, for example.
We summarize properties of the quasi-interpolation operator $\quasiinter$
which we use to establish the reliable error estimate.

\begin{lemma}
	\label{lem:saa:collectionquasiinter}
	If \Cref{ass:domainkappa_1,ass:csph}	 
	hold and $h \in (0,1)$, then 
	\emph{(a)}
	$
	\norm[L^1(\domain)]{\quasiinter u} \leq 
	\norm[L^1(\domain)]{u} 
	$
	for  each $u \in L^1(\domain)$,
	and
	\emph{(b)}
	$\quasiinter  u \in \adcsph$
	for all $u \in \adcsp$.
\end{lemma}%
\begin{proof}
	\textnormal{(a)}
	A proof based on a computation in 
	\cite[p.\ 870]{Wachsmuth2011}
	can be found in \cite[Lem.\ 3.3.6]{Milz2021a}.
	
	\textnormal{(b)}
	Fix $u \in \adcsp$. We have
	$\quasiinter u \in \csph$, as $\domain$ is bounded.
	Using \Cref{ass:csph}, \eqref{eq:adcsp}, 
	and \eqref{eq:saa:quasiinter}, 
	it follows that $\lb \leq \pi_h^j[u] \leq \ub$.
	Using \Cref{ass:csph} onces more, we have
	$\quasiinter u \in \adcsp$.
	Combined with  $\adcsph = \adcsp \cap
	\csph$, we obtain $\quasiinter u \in \adcsph$.
\end{proof}

\section{Reliable error estimates}
\label{sec:error}

\Cref{prob:femsaa} states the reliable  error estimate.
Let $u^*$ be the solution to the
risk-neutral problem \eqref{eq:saa:lqcp}
and let $\uhN$ be the solution to the discretized
SAA problem \eqref{eq:saa:femsaa}.
We demonstrate in \Cref{lem:existence_measurability} the existence
of solutions and the measuarability of $\uhN$.
We define the problem-dependent parameter
\begin{align}
	\label{eq:cstar}
	\mathcal{C}^* = \tfrac{C_\domain^2}{\rdc_{\min}}
	\norm[L^2(\domain)]{u^*}+\norm[L^2(\domain)]{y_d}.
\end{align}

\begin{theorem}
	\label{prob:femsaa}
	Suppose that
	\Cref{ass:domainkappa,ass:2020-11-08T15:31:39.527,%
		ass:csph,ass:2020-01-31T16:12:16.968}
	are fulfilled. 
	Let $\varepsilon > 0$ and let $\delta$, $h \in (0, 1)$. 
	If 
	$N \geq 2\ln(2/\delta)/\varepsilon^2
	$, 
	then with a probability of at least
	$1-\delta$, 
	\begin{subequations}
		\label{eq:reliableestimate}
		\begin{align}
			\label{eq:reliableestimatea}
			\norm[L^2(\domain)]{\uhN-u^*}
			&<
			(1/\alpha)h\Big(8C_\csp^{1/2}\alpha + 17C_\csp\alpha+
			\tfrac{16C_\csp C_\domain^4}{\rdc_{\min}^2}
			\Big)
			\norm[H^1(\domain)]{u^*}
			\\
			\label{eq:reliableestimateb}
			& \quad  + 
			(32/\alpha)\varepsilon 
			\Big(\tfrac{C_\domain^2\mathcal{C}^*}{\rdc_{\min}}
			\Big)
			\\
			\label{eq:reliableestimatec}
			& \quad +
			(4/\alpha) h
			\Big(\tfrac{C_\csp^{1/2}(C_\domain+1)C_\domain\mathcal{C}^*}{\rdc_{\min}}
			\Big)
			\\
			\label{eq:reliableestimated}
			& \quad +
			(32/\alpha)h^2
			\Big(
			C_Y^2 C_{H^2}^2
			\tfrac{\rdc_{\max}^{7/2}}{\rdc_{\min}^{17/2}}
			\rdc_{\max,1}^4
			\mathcal{C}^* \Big).
		\end{align}
	\end{subequations}
\end{theorem}

\Cref{prob:femsaa} is established in \cref{sec:2021-02-22T19:08:49.94}. 
\Cref{prob:femsaa} yields the exponential tail bound in
\eqref{eq:2021-02-22T15:31:07.382} with $c_1$ and $c_2$ being
problem-dependent parameters.
Since $\adcsp$ is bounded, we can establish a bound on 
$\norm[L^2(\domain)]{u^*}$ using \eqref{eq:adcsp}. 
An upper bound on $\norm[H^1(\domain)]{u^*}$ can be derived using
\Cref{lem:saa:2020-11-08T17:14:08.485,lem:saa:H1reg}.
The deterministic parameters in 
the error estimate \eqref{eq:reliableestimate} are as follows:
$C_\domain > 0$ is \friedrichs' constant of the domain $\domain$, 
the parameters 
$\rdc_{\min}$, $\rdc_{\max}$, and
$\rdc_{\max,1}$ are finite by \Cref{ass:domainkappa},
but depend on characteristics of the random diffusion coefficient
$\rdc$, 
$C_{H^2} > 0$ is a constant appearing in an 
$H^2(\domain)$-regularity estimate of the PDE solution
(see \Cref{lem:basicerror}), 
$C_\ssp > 0$ is defined in \Cref{ass:2020-11-08T15:31:39.527}
and depends on the state space discretization, and
$C_\csp > 0$ is defined in \Cref{ass:2020-01-31T16:12:16.968}
and depends on the control space discretization.

\section{Properties of the control problem and proof of the 
	reliable error estimates}
\label{sec:2021-02-22T19:08:49.94}

To establish the reliable error estimate, we define
certain mappings.
We define $\pobj : H_0^1(\domain) \to \real$ 
and $\rpobj : H_0^1(\domain) \times \Xi \to \real$ by 
\begin{align*}
	\pobj(y) = (1/2) \norm[L^2(\domain)]{y-y_d}^2
	\quad \tand \quad 
	\rpobj(u,\xi) = (1/2) \norm[L^2(\domain)]{S(u,\xi)-y_d}^2.
\end{align*}
Furthermore, let us define $\erpobj: L^2(\domain) \to \real$,
$\hat{\erpobj}_N : L^2(\domain) \to \real$
and $\erpobjhN : L^2(\domain) \to \real$
by 
\begin{align*}
	\begin{aligned}
		\erpobj(u) & = \cE{\rpobj(u,\xi)} + (\alpha/2) \norm[L^2(\domain)]{u}^2,
		\;\; \tand \;\; 
		\hat{\erpobj}_N(u) = \frac{1}{N}\sum_{i=1}^N \rpobj(u,\xi^i)
		+ (\alpha/2) \norm[L^2(\domain)]{u}^2,
		\\
		\erpobjhN(u) &= \frac{1}{N}\sum_{i=1}^N \pobj(S_h(u,\xi^i))
		+ (\alpha/2) \norm[L^2(\domain)]{u}^2.
	\end{aligned}
\end{align*}
Since $\xi^1,\xi^2, \ldots$ are defined on the common probability
space $(\Omega, \cF, P)$, 
we can consider the functions $\hat{\erpobj}_N$ and $\erpobjhN$
as mappings defined on
$L^2(\domain) \times \Omega$, but we often omit the second
argument.

To address measurability issues and the existence of solutions,
and to compute derivatives,
we express the PDEs \eqref{eq:S} and \eqref{eq:Sh} as linear systems.
The  linear elliptic PDE \eqref{eq:S} can be written 
equivalently as 
$
A(\xi)y_{\xi} = Bu
$,
where
$A : \Xi \to \sblf[H^{-1}(\domain)]{H_0^1(\domain)}$ 
and $B : L^2(\domain) \to H^{-1}(\domain)$
are defined by
\begin{align*}
	\dualpHzeroone[\domain]{A(\xi)y}{v}
	= \int_{\domain}\rdc(\xi)\nabla y\cdot \nabla v\, \du x
	\;\; \tand \;\;
	\dualpHzeroone[\domain]{Bu}{v}
	=  \inner[L^2(\domain)]{u}{v}.
\end{align*}
Owing to \Cref{ass:domainkappa},
$A$ and $B$ are well-defined  and 
$\norm[\spL{L^2(\domain)}{H_0^1(\domain)}]{B} \leq C_\domain$.
We compute the gradients of $\erpobj$, $\hat{\erpobj}_N$, 
and $\erpobjhN$ using the 
adjoint approach. Therefore, we state the adjoint equation to
the linear elliptic PDE \eqref{eq:S}. 
For each $(u,\xi) \in L^2(\domain) \times \Xi$, 
let the adjoint state $z_{\xi} = z(u,\xi) \in H_0^1(\domain)$ 
be the solution to
the parameterized adjoint equation
\begin{align}
	\label{eq:saa:2020-07-08T15:31:36.424}
	\int_{\domain}\rdc(\xi)\nabla z_{\xi} \cdot \nabla v \, \du x = -
	\int_{\domain} (S(u,\xi)-y_d)v \, \du x
	\quad \tfa \quad v \in H_0^1(\domain).
\end{align}

The discretized PDE \eqref{eq:Sh} can be expressed as
$A_h(\xi)y_{\xi,h}= B_h u$, 
where 
$A_h : \Xi \to \sblf[\ssph^*]{\ssph}$ 
and
$B_h : L^2(\domain) \to \ssph^*$
are defined by
\begin{align*}
	\begin{aligned}
		\dualpb[{\ssph^*,\ssph}]{A_h(\xi)y_h}{v_h} =
		\int_{\domain}\rdc(\xi)\nabla y_{h} \cdot \nabla v_h \, \du x
		\;\; \tand \;\; 
		\dualpb[{\ssph^*,\ssph}]{B_hu}{v_h}
		=  \inner[L^2(\domain)]{u}{v_h}.
	\end{aligned}
\end{align*}
\Cref{ass:2020-11-08T15:31:39.527,ass:domainkappa_1} 
ensure that $A_h$
and $B_h$ are well-defined.
Finally, we state the discretized adjoint equation. 
For each $(u,\xi) \in L^2(\domain) \times \Xi$, 
let the discretized
adjoint state $z_{\xi,h} = z_h(u,\xi) \in \ssph$ be the solution to
the discretized adjoint equation
\begin{align}
	\label{eq:saa:2020-07-08T15:31:36.424h}
	\int_{\domain}\rdc(\xi)\nabla z_{\xi,h} \cdot \nabla v_h  \, \du x = 
	- \int_{\domain} (S_h(u,\xi)-y_d)v_h  \, \du x
	\quad \tfa \quad v_h \in \ssph.
\end{align}

\subsection{Stability estimates}

We formulate stability estimates for the states and adjoint states,
and demonstrate the uniform measurabiliy of the random 
operator $A$ and its inverse. The stability estimates are well-known
(see, e.g., \cite{Charrier2012}) and allow us to establish 
the reliable error estimate \eqref{eq:reliableestimate} 
with problem-dependent parameters made explicit.
Measurability of inverses of measurable operator-valued
mappings between two real separable Banach spaces is a classical topic
(see, e.g., \cite[p.\ 192]{Hans1961}).
Under \Cref{ass:domainkappa}, we demonstrate the uniform measurability
of $A$ and its inverse.

\begin{lemma}
	\label{lem:saa:2020-11-08T17:14:08.485} 
	If \Cref{ass:domainkappa} is satisfied, 
	then the following statements hold.
	\begin{enumerate}[nosep,wide,leftmargin=*,before={\parskip=0pt}]
		\item 
		\label{itm:2020-11-08T15:32:48.054_1111}
		For each $\xi \in \Xi$, $A(\xi)$ is self-adjoint
		and has a bounded inverse.
		The mapping $A$ and its inverse are uniformly measurable.
		
		\item 	
		\label{itm:2020-11-08T15:32:48.054_1}
		For all $(u, \xi) \in L^2(\domain) \times \Xi$, 
		$$
		\eqeqnorm[H^1(\domain)]{S(u, \xi)} \leq (C_\domain/\rdc_{\min})
		\norm[L^2(\domain)]{u}
		\quad \tand \quad 
		\norm[L^2(\domain)]{S(u, \xi)} \leq (C_\domain^2/\rdc_{\min})
		\norm[L^2(\domain)]{u}.
		$$
		
		\item 	
		\label{itm:2020-11-08T19:06:13.137_2}
		For all $(u, \xi) \in L^2(\domain) \times \Xi$, we have
		$z(u, \xi) = -S(S(u, \xi)-y_d, \xi)$,
		\begin{align*}
			\eqeqnorm[H^1(\domain)]{z(u, \xi)} 
			&\leq\tfrac{C_\domain}{\rdc_{\min}}
			\big(\tfrac{C_\domain^2}{\rdc_{\min}}
			\norm[L^2(\domain)]{u}+\norm[L^2(\domain)]{y_d}\big), 
			\quad \tand \\
			\norm[H^1(\domain)]{z(u, \xi)}
			&\leq \tfrac{(C_\domain+1)C_\domain}{\rdc_{\min}}
			\big(\tfrac{C_\domain^2}{\rdc_{\min}}
			\norm[L^2(\domain)]{u}+\norm[L^2(\domain)]{y_d}\big).
		\end{align*}
	\end{enumerate}
\end{lemma}

\begin{proof}
	\textnormal{(a)}
	Fix $\xi \in \Xi$. 
	The computations in  \cite[p.\ 62]{Hinze2009}
	imply that $A(\xi)$ is self-adjoint.
	\Cref{ass:domainkappa} and
	the Lax--Milgram lemma imply that $A(\xi)$ 
	has a bounded inverse
	with $\norm[{\sblf[H_0^1(\domain)]{H^{-1}(\domain)}}]{A(\xi)^{-1}}
	\leq 1/\rdc_{\min}$.
	
	H\"older's inequality implies that
	the mapping 
	$\phi : C^1(\bar{\domain}) \to \sblf[H^{-1}(\domain)]{H_0^1(\domain)}$
	defined by 
	$\dualpHzeroone[\domain]{\phi(\rrdc)y}{v}
	=
	\int_{\domain}\rrdc\nabla y\cdot \nabla v\, \du x$
	is (Lipschitz) continuous.
	Since $A = \phi \circ \kappa$, 
	the uniform measurability of $A$
	follows from the strong measurability
	of $\rdc$ and the composition rule 
	\cite[Cor.\ 1.1.11]{Hytoenen2016}.
	The set, $O$, of invertible maps in 
	$\sblf[H^{-1}(\domain)]{H_0^1(\domain)}$ is open
	\cite[Thm.\ 5.8]{Alt2016}
	and $O \ni T \mapsto T^{-1}$ is continuous.
	Hence $A^{-1}$ is the composition of a continuous
	function and $A$. Hence $A^{-1}$  
	is uniformly measurable \cite[p.\ 7]{Hytoenen2016}.
	
	\textnormal{(b)}
	The first bound follows from 
	\cite[eq.\ (2.1)]{Charrier2012}.
	The second bound is implied by the first
	one and \friedrichs' inequality.
	
	\textnormal{(c)} 
	Using \eqref{eq:S} and
	\eqref{eq:saa:2020-07-08T15:31:36.424},
	we obtain
	$z(u, \xi) = -S(S(u, \xi)-y_d, \xi)$.
	Combined with part~\ref{itm:2020-11-08T15:32:48.054_1}, 
	we obtain the first stability estimate.
	\friedrichs'
	inequality  implies
	$\norm[H^1(\domain)]{z(u, \xi)}
	\leq (C_\domain+1)\eqeqnorm[H^1(\domain)]{z(u, \xi)} $.
	Hence the first stability estimate implies the second one.
\end{proof}

\Cref{lem:saa:2020-11-08T17:14:08.485_2}  
establishes results similar to those in 
\Cref{lem:saa:2020-11-08T17:14:08.485}, but for the discretized PDEs.
\begin{lemma}
	\label{lem:saa:2020-11-08T17:14:08.485_2} 
	If \Cref{ass:domainkappa,ass:2020-11-08T15:31:39.527} hold and $h \in (0,1)$,
	then the following statements hold.
	\begin{enumerate}[nosep,wide,leftmargin=*]
		\item 
		\label{itm:2020-11-08T15:32:48.054_1111_2}
		For each $\xi \in \Xi$, $A_h(\xi)$ is self-adjoint
		and has a bounded inverse.
		The mapping $A_h$  and its inverse are uniformly measurable.
		
		\item 	
		\label{itm:2020-11-08T15:32:48.054_1_2}
		For all $(u, \xi) \in L^2(\domain) \times \Xi$, 
		$$
		\eqeqnorm[H^1(\domain)]{S_h(u, \xi)} \leq (C_\domain/\rdc_{\min})
		\norm[L^2(\domain)]{u}
		\;\; \tand \;\;  
		\norm[L^2(\domain)]{S_h(u, \xi)} \leq (C_\domain^2/\rdc_{\min})
		\norm[L^2(\domain)]{u}.
		$$
		\item 	
		\label{itm:2020-11-08T19:06:13.137_2_2}
		For all $(u, \xi) \in L^2(\domain) \times \Xi$, we have
		$z_h(u, \xi) = -S_h(S_h(u, \xi)-y_d, \xi)$.
	\end{enumerate}
\end{lemma}

\begin{proof}
	The assertions can be established using arguments similar to 
	those used  in 
	the proof of \Cref{lem:saa:2020-11-08T17:14:08.485}.
\end{proof}

\subsection{Differentiability, existence of solutions, and 
	optimality conditions}

We establish the continuous
differentiability of $\erpobj$, $\hat{\erpobj}_N$ and $\erpobjhN$
using  the adjoint approach and \cite[Lem.\ 2.1]{Martin2021a},
comment on the existence of solutions and 
state first-order necessary optimality conditions. 
These basic facts are used in subsequent sections and are 
essential for our proof of \Cref{prob:femsaa}.

\begin{lemma}
	\label{lem:differentiability}
	Let \Cref{ass:domainkappa,ass:2020-11-08T15:31:39.527,%
		ass:csph} 
	hold. 
	Then, the following statements hold true.
	\begin{enumerate}[nosep,wide,leftmargin=*]
		\item 
		The mappings
		$S$, $S_h$, $z$, and $z_h$ are \Caratheodory\ mappings. 
		
		\item 
		\label{eq:cEzu}
		For each $u \in L^2(\domain)$, 
		$\cE{z(u,\xi)} \in H_0^1(\domain)$.

		\item 
		\label{itm:gradients}
		The functions $\erpobj$, $\hat{\erpobj}_N$
		and $\erpobjhN$
		are strongly convex with parameter
		$\alpha$ and
		continuously differentiable.
		For all $u \in L^2(\domain)$,
		$\nabla \erpobj(u) = \alpha u - B^*\cE{z(u,\xi)}$,
		\begin{align*}
			\nabla \hat{\erpobj}_N(u)
			= \alpha u - B^*\bigg[\frac{1}{N}\sum_{i=1}^N 
			z(u,\xi^i)\bigg],
			\quad \tand \quad 
			\nabla \erpobjhN(u)
			= \alpha u
			- B_h^*\bigg[\frac{1}{N}\sum_{i=1}^N z_h(u,\xi^i)\bigg].
		\end{align*}
		\item For each $u \in L^2(\domain)$,
		$B^*v = v$ for all $v \in H_0^1(\domain)$
		and $B_h^*v_h = v_h$
		for all $v_h \in \ssph$.
		\item 
		The function $\erpobjhN $ is a \Caratheodory\
		function.
	\end{enumerate}
\end{lemma}%
\begin{proof}
	\textnormal{(a)}
	We verify the statements for $S$ and $z$ only.
	We have $S(u,\xi) = A(\xi)^{-1}Bu$
	and $z(u,\xi) = -S(S(u,\xi)-y_d,\xi)$;
	see \Cref{lem:saa:2020-11-08T17:14:08.485}.
	\Cref{lem:saa:2020-11-08T17:14:08.485} further
	ensures the uniform measurability of $A^{-1}$. 
	Combined with composition rule
	\cite[Prop.\ 1.1.28]{Hytoenen2016}, we conclude
	that 
	$S(u,\cdot) = A(\cdot)^{-1}Bu$ and hence $z(u,\cdot)$
	are measurable for each $u \in L^2(\domain)$. For each $\xi \in \Xi$, 
	\Cref{lem:saa:2020-11-08T17:14:08.485} also ensures that
	$S(\cdot, \xi)$ and $z(\cdot, \xi)$ are continuous.

	\textnormal{(b)}
	\Cref{lem:saa:2020-11-08T17:14:08.485} 
	and part~(a) ensure
	that $z(u,\cdot)$ is measurable and 
	$\cE{\eqeqnorm[H^1(\domain)]{z(u, \xi)}} < \infty$.
	Hence
	$\cE{z(u, \xi)}\in H_0^1(\domain)$
	\cite[p.\ 14]{Hytoenen2016}.

	\textnormal{(c)}
	The statements for $\hat{\erpobj}_N$
	and $\erpobjhN$ are a consequence of the adjoint approach
	\cite[sect.\ 1.6.2]{Hinze2009}, the fact
	that $\csph$ and $\ssph$ are Hilbert spaces, and
	\Cref{lem:saa:2020-11-08T17:14:08.485,%
		lem:saa:2020-11-08T17:14:08.485_2}.
	The \gateaux\ differentiability of $\erpobj$ 
	and the gradient formula are implied by \cite[Lem.\ 2.1]{Martin2021a}
	(see also \cite[p.\ 989]{Martin2021a}
	and \cite[Lem.\ C.3]{Geiersbach2020})
	as well as the fact that
	$\cE{B^*z(u,\xi)} = B^*\cE{z(u,\xi)}$ for all $u \in L^2(\domain)$.
	This identity is implied by 
	part~\ref{eq:cEzu}, 
	the linearity and boundedness of $B^*$,
	and the definition of the Bochner integral
	\cite[eq.\ (1.2)]{Hytoenen2016}.
	Using \Cref{lem:saa:2020-11-08T17:14:08.485,%
		lem:saa:2020-11-08T17:14:08.485_2}
	and the dominated convergence theorem, 
	we can show the continuous differentiability
	of $\erpobj$.
	The strong convexity is a result of the control
	regularization $(\alpha/2)\norm[L^2(\domain)]{\cdot}^2$
	and the fact that $\csph \subset L^2(\domain)$ is a Hilbert space.
	
	\textnormal{(d)}
	Since the adjoint operator $B^*$ of $B$ is given by
	$\inner[L^2(\domain)]{B^*v}{u} = \inner[L^2(\domain)]{v}{u}$
	for $u \in L^2(\domain)$ and $v \in H_0^1(\domain)$
	(cf.\ \cite[p.\ 62]{Hinze2009}), 
	we obtain  the first identity.
	We recall that $L^2(\domain)$ is identified with 
	$L^2(\domain)^*$ and $(\ssph^*)^*$ is identified with $\ssph$.
	Let us define
	$\iota_h \in \spL{\ssph}{L^2(\domain)}$ by
	$\inner[L^2(\domain)]{\iota_hv_h}{u} 
	= \inner[L^2(\domain)]{v_h}{u}$.
	Since $\ssph$ is a subspace of $L^2(\domain)$, we have
	$\iota_h^* = B_h$ (cf.\ \cite[p.\ 21]{Bonnans2013}).
	Hence $(\iota_h^*)^* = B_h^*$. 
	Combined with $(\iota_h^*)^* = \iota_h$
	(cf.\ \cite[p.\ 390]{Alt2016}), we have $\iota_h = B_h^*$.
	
	\textnormal{(e)}
	The mapping
	$(u,\xi) \mapsto \pobj(S_h(u,\xi))$ is a composition of 
	the continuous function $H_0^1(\domain) \ni y \mapsto \pobj(y)$
	with the \Caratheodory\
	mapping $(u,\xi) \mapsto S_h(u,\xi) \in \ssph$. 
	Since $\ssph \subset H_0^1(\domain)$ is a closed subspace, 
	the function $\erpobjhN$  is a \Caratheodory\
	function.
\end{proof}

\Cref{lem:existence_measurability} establishes the existence of solutions
using standard arguments.

\begin{lemma}
	\label{lem:existence_measurability}
	If \Cref{ass:domainkappa,ass:2020-11-08T15:31:39.527,%
		ass:csph}  hold, 
	then \emph{(a)} the control problem \eqref{eq:saa:lqcp} 
	has a unique solution $u^*$, 
	\emph{(b)}
	for each $\omega \in \Omega$,
	the discretized SAA problem \eqref{eq:saa:femsaa} 
	has a unique  solution $\uhN(\omega)$, and
	\emph{(c)}
	$\uhN$ is measurable.
\end{lemma}%
\begin{proof}
	(a)--(b) The existence 
	and uniqueness can be established
	using \cite[Lem.\ 2.33]{Bonnans2013}.
	
	\textnormal{(c)}
	Since $\erpobjhN$ is a \Caratheodory\ function
	(see \Cref{lem:differentiability}), 
	$(\Omega, \cF, P)$ is complete, 
	and $\norm[L^1(\domain)]{\cdot}$ is continuous, 
	$\uhN : \Omega \to \csph$ 
	is measurable \cite[Thm.\ 8.2.11]{Aubin2009}.
	Since $\csph$ is a closed subspace of $L^2(\domain)$, 
	$\uhN : \Omega \to L^2(\domain)$ is also measurable.
\end{proof}
The next lemma provides 
consequences of
first-order necessary optimality conditions.
\begin{lemma}
	If \Cref{ass:domainkappa,ass:2020-11-08T15:31:39.527,%
		ass:csph} hold, then
	\begin{align}
		\label{eq:cfonoc}
		\begin{aligned}
			\innerp{L^2(\domain)}{\nabla \erpobj(u^*)}
			{\uhN-u^*} + \gamma\norm[L^1(\domain)]{\uhN}
			- \gamma\norm[L^1(\domain)]{u^*} &\geq 0,
			\\
			\innerp{L^2(\domain)}{\nabla \erpobjhN(\uhN)}
			{\quasiinter u^* -\uhN} + 
			\gamma\norm[L^1(\domain)]{\quasiinter u^*}
			- \gamma\norm[L^1(\domain)]{\uhN} &\geq 0.
		\end{aligned}
	\end{align}
\end{lemma}%
\begin{proof}
	\Cref{lem:differentiability} ensures that  $\erpobj$
	and $\erpobjhN$ are continuously differentiable and convex.
	\Cref{lem:saa:collectionquasiinter} 
	ensures $\quasiinter u^* \in \adcsph$
	and $\adcsph = \csph \cap \adcsp$ yields
	$\uhN \in \adcsp$.
	Hence the inequalities follow from the 
	optimality conditions derived in 
	\cite[Thm.\ 4.42]{Ito2008}.
\end{proof}

\subsection{Regularity of the solution}

We show that the solution $u^*$ to \eqref{eq:saa:lqcp}
has square integrable weak derivatives
and provide a bound on the solution's weak derivatives.
Our derivation is based on standard arguments
used to establish regularity of solutions to deterministic
PDE-constrained optimization problems
\cite{Borzi2005,Casas2012a,Falk1973,Geveci1979,%
	Malanowski1982,Troeltzsch2010a,Wachsmuth2011}.

\begin{lemma}
	\label{lem:saa:H1reg}
	If \Cref{ass:domainkappa} holds, 
	then 
	$u^* $,
	$\nabla \erpobj(u^*) \in H^1(\domain)$,
	and 
	$$
	\norm[H^1(\domain)]{u^*}
	\leq (1/\alpha)\norm[H^1(\domain)]{\cE{z(u^*, \xi)}}
	+ (|\lb|+\ub)\norm[L^2(\domain)]{1}.
	$$
\end{lemma}%
\begin{proof}
	Using \Cref{lem:differentiability}
	and the optimality of $u^*$, we obtain the  optimality
	condition
	$u^* = \prox{\psi/\alpha}{(1/\alpha)B^*\cE{z(u^*,\xi)}}$
	(cf.\ \cite[p.\ 2092]{Mannel2020}),
	where $\psi(\cdot) = \gamma \norm[L^1(\domain)]{\cdot} + I_{\adcsp}(\cdot)$.
	\Cref{lem:differentiability} yields $\cE{z(u^*,\xi)}\in H_0^1(\domain)$.
	Combined with $B^*v = v$ valid for all $v \in H_0^1(\domain)$
	and \Cref{lem:Feb0420211008}, we obtain $u^* \in H^1(\domain)$
	and the stability estimate.
	\Cref{lem:differentiability} further ensures
	$\nabla \erpobj(u^*) \in H^1(\domain)$.
\end{proof}%

\subsection{Basic error estimates}
We state basic error estimates which are direct consequences
of those derived in \cite{Ali2017,Charrier2013}
(see also \cite{Guth2019}).

\begin{lemma}
	\label{lem:basicerror}
	Let \Cref{ass:domainkappa,ass:2020-11-08T15:31:39.527}
	hold and let $h \in (0,1)$.
	Then, the following statements hold.
	\begin{enumerate}[wide,nosep,leftmargin=*]
		\item 
		\label{itm:basicerror1}
		There exists $C_{H^2} > 0$
		such that
		$
		\norm[H^2(\domain)]{S(u, \xi)}
		\leq C_{H^2}\newconstant{\label{H2}} \norm[L^2(\domain)]{u}
		$
		and $S(u, \xi) \in H^2(\domain)$
		for all $(u, \xi) \in L^2(\domain) \times \Xi$,
		where 
		\begin{align*}
			\oldconstant{H2} = (\rdc_{\max}/\rdc_{\min}^4)\rdc_{\max,1}^2.
		\end{align*}

		\item 
		\label{itm:2020-11-08T15:32:48.054_4}
		$
		\eqeqnorm[H^1(\domain)]{S(u,\xi)-S_h(u, \xi)}
		\leq C_\ssp C_{H^2}\newconstant{\label{H1h}} h\norm[L^2(\domain)]{u}
		$
		for all $(u, \xi) \in L^2(\domain) \times \Xi$,
		where
		$$\oldconstant{H1h} = (\rdc_{\max}^{3/2}/\rdc_{\min}^{9/2})
		\rdc_{\max,1}^2
		= \oldconstant{H2}(\rdc_{\max}/\rdc_{\min})^{1/2}.$$
		
		\item 
		\label{itm:2020-11-08T15:32:48.054_5}
		$
		\norm[L^2(\domain)]{S(u,\xi)-S_h(u, \xi)}
		\leq C_Y^2 C_{H^2}^2 \newconstant{\label{L2h2}}
		h^2\norm[L^2(\domain)]{u} 
		$
		for all $(u, \xi) \in L^2(\domain) \times \Xi$,
		where
		\begin{align*}
			\oldconstant{L2h2} = (\rdc_{\max}^{7/2}/\rdc_{\min}^{17/2}) 
			\rdc_{\max,1}^4 = \rdc_{\max}\oldconstant{H2}\oldconstant{H1h}.
		\end{align*}

		\item 
		$\norm[L^2(\domain)]{\nabla \erpobjhN(u^*)-\nabla \hat{\erpobj}_{N}(u^*)}
		\leq 
		\newconstant{\label{gradL2h2}} h^2\big(
		(C_\domain^2/\rdc_{\min})
		\norm[L^2(\domain)]{u^*} + \norm[L^2(\domain)]{y_d}\big)$, 
		where 
		$$
		\oldconstant{gradL2h2} = 	
		2C_Y^2 C_{H^2}^2 (\rdc_{\max}^{7/2}/\rdc_{\min}^{17/2}) 
		\rdc_{\max,1}^4 = 
		2C_Y^2 C_{H^2}^2  \oldconstant{L2h2}.
		$$
	\end{enumerate}
\end{lemma}

\begin{proof}
	\textnormal{(a)} Applying \cite[Thm.\ 3.1]{Ali2017}, we obtain the
	statements.
	
	\textnormal{(b)}
	The assertion follows  from the
	proof of \cite[Thm.\ 3.9]{Charrier2013}. 
	\cea's lemma used in the
	proof of \cite[Thm.\ 3.9]{Charrier2013}
	remains valid  under our assumptions
	because $\ssph$ is a closed subspace of $H_0^1(\domain)$.

	\textnormal{(c)} The assertion follows  from 
	the arguments used in the proof of \cite[Thm.\ 4.4]{Ali2017}, 
	but we apply \Cref{ass:domainkappa}
	instead of 
	\cite[Lem.\ 4.3]{Ali2017} to establish the assertion.
	The Galerkin orthogonality used in the proof of 
	\cite[Thm.\ 4.4]{Ali2017} is valid under our assumptions 
	because $\ssph$ is a closed subspace of $H_0^1(\domain)$.
	
	\textnormal{(d)}
	Using \Cref{lem:saa:2020-11-08T17:14:08.485,%
		lem:saa:2020-11-08T17:14:08.485_2,%
		lem:differentiability}, 
	we find that for all $v \in L^2(\domain)$,
	\begin{align*}
		& \inner[L^2(\domain)]
		{ \nabla \erpobjhN(u^*)-\nabla \hat{\erpobj}_{N}(u^*)}{v}
		\\
		&\quad 
		= \frac{1}{N} \sum_{i=1}^N
		\big( S_h(S_h(u^*, \xi^i)-y_d, \xi^i)
		-S(S(u^*, \xi^i)-y_d, \xi^i), v\big)_{L^2(\domain)}.	
	\end{align*}
	We separately estimate
	$S_h(S(u^*, \xi)-y_d, \xi)-S(S(u^*, \xi)-y_d, \xi)$
	and		
	$S_h(S_h(u^*, \xi)-y_d, \xi)-S_h(S(u^*, \xi)-y_d, \xi)$
	for $\xi \in \Xi$.
	Using part~\ref{itm:2020-11-08T15:32:48.054_5}
	and \Cref{lem:saa:2020-11-08T17:14:08.485_2}, we  find
	that
	\begin{align*}
		\norm[L^2(\domain)]
		{S_h(S_h(u^*, \xi)-y_d, \xi)-S_h(S(u^*, \xi)-y_d, \xi)}
		& \leq 
		(C_\domain^2/\rdc_{\min})
		\norm[L^2(\domain)]{S_h(u^*, \xi)-S(u^*, \xi)}
		\\
		&\leq 
		(C_\domain^2/\rdc_{\min}) C_Y^2 C_{H^2}^2 \oldconstant{L2h2}
		h^2\norm[L^2(\domain)]{u^*}.
	\end{align*}
	Using part~\ref{itm:2020-11-08T15:32:48.054_5}
	and \Cref{lem:saa:2020-11-08T17:14:08.485}, we further  have
	\begin{align*}
		\norm[L^2(\domain)]
		{S_h(S(u^*, \xi)-y_d, \xi)-S(S(u^*, \xi)-y_d, \xi)}
		& \leq  C_Y^2 C_{H^2}^2  \oldconstant{L2h2}
		h^2\norm[L^2(\domain)]{S(u^*, \xi)-y_d}.
	\end{align*}
	The triangle inequality
	and \Cref{lem:saa:2020-11-08T17:14:08.485} also yield
	$$
	\norm[L^2(\domain)]{S(u^*, \xi)-y_d}
	\leq 
	(C_\domain^2/\rdc_{\min})
	\norm[L^2(\domain)]{u^*}
	+
	\norm[L^2(\domain)]{y_d}.
	$$
	Putting together the pieces, we obtain the assertion.
\end{proof}

We collect further basic error estimates
based on those established in \cite{Martin2021,Martin2021a,Wachsmuth2011}.

\begin{lemma}
	\label{lem:basicerror_2}
	Let \Cref{ass:domainkappa,ass:2020-11-08T15:31:39.527,%
		ass:csph,ass:2020-01-31T16:12:16.968}
	hold and let $h \in (0,1)$.
	Then the following statements hold.
	\begin{enumerate}[nosep,wide,leftmargin=*,before={\parskip=0pt}]
		\item  
		We have
		$
		\norm[L^2(\domain)]{\nabla \erpobjhN(\quasiinter u^*)-
			\nabla \erpobjhN( u^*)}
		\leq 
		\big(\alpha + \tfrac{C_\domain^4}{\rdc_{\min}^2}\big)
		\norm[L^2(\domain)]{\quasiinter u^*-u^*}
		$.
		\item It holds that
		$|\innerp{L^2(\domain)}{\nabla \erpobj(u^*)}
		{\quasiinter u^*-u^*}|
		\leq C_{\csp} h^2 \norm[H^1(\domain)]{\nabla \erpobj(u^*)}
		\norm[H^1(\domain)]{u^*}
		$.
	\end{enumerate}
\end{lemma}

\begin{proof}
	\textnormal{(a)}
	The estimate can be established using computations
	similar to those used to prove \cite[Lem.\ 3.5]{Martin2021}
	(see also \cite[pp.\ 987 and 989, and Lem.\ 2.2]{Martin2021a}).
	
	\textnormal{(b)}
	Since 
	$H^1(\domain) \embedding L^2(\domain) \embedding H^1(\domain)^*$
	is a Gelfand triple \cite[p.\ 147]{Troeltzsch2010a}, 
	the embedding $L^2(\domain) \embedding H^1(\domain)^*$
	is given by
	$\dualp[H^1(\domain)]{v}{w} = \innerp{L^2(\domain)}{v}{w}$
	for all $v \in L^2(\domain)$
	and $w \in H^1(\domain)$
	\cite[Rem.\ 1.17]{Hinze2009}.
	Combined with 	
	$u^*$, $\nabla \erpobj(u^*) \in H^1(\domain)$
	(see \Cref{lem:saa:H1reg}), and
	$\quasiinter u^* \in \adcsph \subset L^2(\domain)$
	(see \Cref{lem:saa:collectionquasiinter}),
	we have
	$$|\innerp{L^2(\domain)}{\nabla \erpobj(u^*)}
	{\quasiinter u^*-u^*}|
	\leq \norm[H^1(\domain)]{\nabla \erpobj(u^*)}
	\norm[H^1(\domain)^*]{\quasiinter u^*-u^*}.
	$$ 
	Together with 
	\Cref{ass:2020-01-31T16:12:16.968}, we obtain the 
	estimate.
\end{proof}

\Cref{prop:saa:2020-11-14T21:32:40.594} establishes an upper bound on
$\alpha \norm[L^2(\domain)]{\quasiinter u^* - \uhN}^2$.

\begin{proposition}
	\label{prop:saa:2020-11-14T21:32:40.594}
	If \Cref{ass:domainkappa,ass:2020-11-08T15:31:39.527,%
		ass:csph,ass:2020-01-31T16:12:16.968} hold,
	then 
	\begin{align}
		\label{eq:saa:2020-11-07T23:12:24.196}
		\begin{aligned}
			\alpha \norm[L^2(\domain)]{\quasiinter u^* - \uhN}^2
			&\leq 
			\innerp{L^2(\domain)}{\nabla \erpobjhN(\quasiinter u^*)-
				\nabla \erpobjhN( u^*)}
			{\quasiinter u^*-\uhN}
			\\
			& \quad +
			\innerp{L^2(\domain)}{\nabla \erpobjhN(u^*)-\nabla \hat{\erpobj}_{N}(u^*)}
			{\quasiinter u^*-\uhN}
			\\
			& \quad +
			\innerp{L^2(\domain)}{\nabla \hat{\erpobj}_N(u^*)-\nabla \erpobj(u^*)}
			{\quasiinter u^* - \uhN}
			\\
			& \quad+
			\innerp{L^2(\domain)}{\nabla \erpobj(u^*)}
			{\quasiinter u^*-u^*}.
		\end{aligned}
	\end{align}
\end{proposition}

\begin{proof}
	The proof is inspired 
	by the arguments used in
	\cite[Thm.\ 5.2]{Meidner2008}.
	Adding the inequalities in \eqref{eq:cfonoc} and using
	$\norm[L^1(\domain)]{\quasiinter u^*} \leq \norm[L^1(\domain)]{u^*}$
	(see \Cref{lem:saa:collectionquasiinter})
	yields
	\begin{align}
		\label{eq:saa:2020-11-07T22:25:52.42}
		0 \leq 
		& \innerp{L^2(\domain)}{\nabla \erpobj(u^*)}
		{\uhN-u^*} 
		+\innerp{L^2(\domain)}{\nabla \erpobjhN(\uhN)}
		{\quasiinter u^* -\uhN}.
	\end{align}
	Since $\erpobjhN$ is strongly convex
	with parameter $\alpha > 0$
	and \gateaux\ differentiable on $L^2(\domain)$
	(see \Cref{lem:differentiability}),
	we have
	\begin{align*}
		\alpha \norm[L^2(\domain)]{\uhN-\quasiinter u^*}^2
		\leq 
		&\innerp{L^2(\domain)}{\nabla \erpobjhN(\quasiinter u^*)-
			\nabla \erpobjhN(\uhN)}
		{\quasiinter u^*-\uhN}.
	\end{align*}
	Adding this inequality 
	and the estimate \eqref{eq:saa:2020-11-07T22:25:52.42}, 
	we conclude that	
	\begin{align*}
		\alpha \norm[L^2(\domain)]{\uhN-\quasiinter u^*}^2
		&\leq 
		\innerp{L^2(\domain)}{\nabla \erpobjhN(\quasiinter u^*)}
		{\quasiinter u^*-\uhN}
		+ \innerp{L^2(\domain)}{\nabla \erpobj(u^*)}
		{\uhN-u^*}
		\\
		& = 
		\innerp{L^2(\domain)}{\nabla \erpobjhN(\quasiinter u^*)}
		{\quasiinter u^*-\uhN}
		+ \innerp{L^2(\domain)}{\nabla \erpobj(u^*)}
		{\quasiinter u^*-u^*}
		\\
		& \quad + \innerp{L^2(\domain)}{\nabla \erpobj(u^*)}
		{\uhN-\quasiinter u^*}.
	\end{align*}
	Manipulating the terms in the right-hand side, we obtain
	\eqref{eq:saa:2020-11-07T23:12:24.196}.
\end{proof}

The gradients in \eqref{eq:saa:2020-11-07T23:12:24.196}
are evaluated at the deterministic
controls $u^*$ and $\quasiinter u^*$, but not at the
random control $\uhN$. This fact is used for our error estimation.
Furthermore, under the hypotheses
of \Cref{prop:saa:2020-11-14T21:32:40.594}, the estimate
\eqref{eq:saa:2020-11-07T23:12:24.196} yields
\begin{align}
	\label{eq:simpleerrorbound}
	\begin{aligned}
		\alpha \norm[L^2(\domain)]{\quasiinter u^* - \uhN}^2
		\leq 
		& \norm[L^2(\domain)]{\nabla \erpobjhN(\quasiinter u^*)-
			\nabla \erpobj( u^*)}
		\norm[L^2(\domain)]{\quasiinter u^*-\uhN}
		\\
		& \quad+
		\innerp{L^2(\domain)}{\nabla \erpobj(u^*)}{\quasiinter u^*-u^*}.
	\end{aligned}
\end{align}
Hence $\norm[L^2(\domain)]{\quasiinter u^* - \uhN}$
can be estimated using 
bounds on $\norm[L^2(\domain)]{\nabla \erpobjhN(\quasiinter u^*)-
	\nabla \erpobj( u^*)}$
and $|\innerp{L^2(\domain)}{\nabla \erpobj(u^*)}
{\quasiinter u^*-u^*}|^{1/2}$.

\subsection{Exponential tail bounds for SAA gradients}
\label{subsec:exponentialtailbound}

We establish an exponential tail bound
for $\nabla \hat{\erpobj}_N(u^*)-\nabla \erpobj(u^*)$
using the large deviation-type bound  established in  
\cite[Thm.\  3]{Pinelis1992} (see also \cite[Thm.\ 3.5]{Pinelis1994}).
Since
$\nabla \hat{\erpobj}_N(u^*)-\nabla \erpobj(u^*)
= B^* \big(\cE{z(u^*,\xi)}- (1/N)\sum_{i=1}^N z(u^*,\xi^i)\big)$, 
the accuracy and reliability of 
$\nabla \hat{\erpobj}_N(u^*)$ as an estimator
for $\nabla \erpobj(u^*)$
is determined by that of the 
adjoint state's sample average.	
Let $\tau$ be positive scalar. 
Suppose that $\rv_1, \rv_2, \ldots$
are independent, mean-zero $L^2(\domain)$-valued random
vectors defined on a common probability space such that for each
$i \in \natural$,
with probability one, $\norm[L^2(\domain)]{\rv_i} \leq \tau$. 
Then for each $N \in \natural$,
\begin{align}
	\label{eq:saa:bOgutQ6IqE}
	\Prob{\norm[L^2(\domain)]{\rv_1+\cdots+\rv_N}\geq \varepsilon}
	\leq 2\exp(-\varepsilon^2/(2N\tau^2))
	\quad \tfa \quad \varepsilon > 0;
\end{align}
see  \cite[Thm.\  3]{Pinelis1992} and \cite[Thm.\ 3.5]{Pinelis1994}.
Using \eqref{eq:saa:bOgutQ6IqE}, we establish a
tail bound for SAA gradients.
\begin{lemma}
	\label{thm:saa:bOgutQ6IqE}
	Let $\delta \in (0,1)$ and let $\varepsilon > 0$.
	If \Cref{ass:domainkappa} holds 
	and $N \geq 2\ln(2/\delta)/\varepsilon^2$, 
	then with a probability of at least $1-\delta$,
	\begin{align*}
		\norm[L^2(\domain)]{\nabla \hat{\erpobj}_N(u^*)-\nabla \erpobj(u^*)}
		< \varepsilon\tau,
		\quad \text{where} \quad 
		\tau =  \tfrac{2C_\domain^2}{\rdc_{\min}}
		\big(\tfrac{C_\domain^2}{\rdc_{\min}}
		\norm[L^2(\domain)]{u^*}+\norm[L^2(\domain)]{y_d}\big).
	\end{align*}
\end{lemma}
\begin{proof}
	The random vectors
	$\rv_i = \nabla_u \rpobj(u^*,\xi^i) - \nabla \erpobj(u^*)$
	are independent
	because $\xi^1, \xi^2, \ldots$ are independent identically distributed
	$\Xi$-valued random elements and $\nabla_u\rpobj(u^*,\cdot)$ is measurable.
	\Cref{lem:saa:2020-11-08T17:14:08.485,lem:differentiability} ensure that
	each $\rv_i$ is integrable and has zero mean. 
	Using 	\Cref{lem:differentiability,lem:saa:2020-11-08T17:14:08.485}
	and \friedrichs' inequality, we have for each $\xi \in \Xi$, 
	\begin{align*}
		\norm[L^2(\domain)]
		{\nabla_u \rpobj(u^*, \xi)-\nabla \erpobj(u^*)}
		& =
		\norm[L^2(\domain)]
		{B^*z(u^*, \xi)-B^*\cE{z(u^*, \xi)}}
		\\
		\nonumber 
		& \leq 
		C_\domain\eqeqnorm[H^1(\domain)]
		{z(u^*, \xi)}
		+C_\domain \cE{\eqeqnorm[H^1(\domain)]{z(u^*, \xi)}}
		\\
		\nonumber
		& \leq \tfrac{2C_\domain^2}{\rdc_{\min}}
		\big(\tfrac{C_\domain^2}{\rdc_{\min}}
		\norm[L^2(\domain)]{u^*}+\norm[L^2(\domain)]{y_d}\big)
		= \tau. 
	\end{align*}
	Hence for each
	$i \in \natural$, 
	with probability one, $\norm[L^2(\domain)]{\rv_i} \leq \tau$.
	Using  \eqref{eq:saa:bOgutQ6IqE}
	and $N \geq 2\ln(2/\delta)/\varepsilon^2$, we obtain the assertion.
\end{proof}

The exponential tail bound \eqref{eq:saa:bOgutQ6IqE} remains
valid for $L^2(\domain)$-valued
random vectors other than essentially bounded ones.
If 	$\varrho > 0$ and $\rv_1, \rv_2, \ldots$
are independent, mean-zero $L^2(\domain)$-valued random
vectors such that for each
$i \in \natural$, 
$\cE{\exp(\varrho^{-2}\norm[L^2(\domain)]{\rv_i}^2)} 
\leq \eu$, 
then \eqref{eq:saa:bOgutQ6IqE} holds true with $\tau$ replaced
by a constant proportional to $\varrho$; see 
\cite[Thm.\ 3]{Pinelis1986} and \cite[Thm.\ 1]{Milz2021}.
While the condition 
$\cE{\exp(\varrho^{-2}\norm[L^2(\domain)]{\rv_i}^2)} \leq \eu$
allows for a larger class of random vectors than essentially bounded
ones, the random vectors 
$\rv_i$ considered in the proof of \Cref{thm:saa:bOgutQ6IqE}
generally violate this condition if the random field $\kappa$ is 
lognormal (see, e.g., \cite[p.\ 87]{Milz2021a}).
This fact has been our main motivation to work with 
\Cref{ass:domainkappa_2}.

In the proof of \Cref{thm:saa:bOgutQ6IqE}, 
we applied  \eqref{eq:saa:bOgutQ6IqE} to obtain a large deviation-type bound
for SAA gradients. This bound depends on the constant $\tau$
defined in \Cref{thm:saa:bOgutQ6IqE}
which provides an almost sure upper bound on $\norm[L^2(\domain)]
{\nabla_u \rpobj(u^*, \xi)-\nabla \erpobj(u^*)}$.
As an alternative to using \eqref{eq:saa:bOgutQ6IqE},
we may use the Bernstein-type inequality 
\cite[Cor.\ 1]{Pinelis1986} to analyze the exponential
tail behavior of SAA gradients. 
We refer the reader to \cite[pp.\  24 and 26]{Massart2007}
for  discussions on Bernstein's inequality for real-valued random variables.

\subsection{Proof of reliable error estimates}
In this section, we establish \Cref{prob:femsaa}.
\begin{proof}[Proof of \Cref{prob:femsaa}]
	\Cref{lem:saa:H1reg} ensures
	$u^* \in H^1(\domain)$. 
	Combined with the triangle inequality, 
	$\eqeqnorm[H^1(\domain)]{u^*} \leq \norm[H^1(\domain)]{u^*}$,
	and \Cref{ass:2020-01-31T16:12:16.968}, we have
	\begin{align}
		\label{eq:reliableestimate_a}
		\norm[L^2(\domain)]{\uhN-u^*}
		\leq C_\csp h \norm[H^1(\domain)]{u^*}
		+ \norm[L^2(\domain)]{\uhN-\quasiinter u^*},
	\end{align}
	where $\quasiinter$ is defined in 
	\eqref{eq:saa:quasiinter}.
	The first addend in \eqref{eq:reliableestimate_a} will contribute
	to the term in \eqref{eq:reliableestimatea}.
	
	In the following steps, we derive a bound on 
	$\norm[L^2(\domain)]{\uhN-\quasiinter u^*}$
	using
	\Cref{lem:saa:2020-11-08T17:14:08.485,lem:basicerror,%
		prop:saa:2020-11-14T21:32:40.594,thm:saa:bOgutQ6IqE,%
		lem:basicerror_2}.	Our
	derivations use the inequality
	$(\rho_1\rho_2)^{1/2} 
	\leq (2\rho_1\rho_2)^{1/2} \leq \rho_1+\rho_2$
	valid for all $\rho_1$, $\rho_2 \geq 0$.
	Using the Cauchy--Schwarz inequality, we find that
	\begin{align*}
		\begin{aligned}
			&|\innerp{L^2(\domain)}{\nabla \erpobjhN(u^*)-\nabla \hat{\erpobj}_{N}(u^*)}
			{\quasiinter u^*-\uhN}|^{1/2}
			\\& \quad \leq 
			(4/\alpha^{1/2})
			\norm[L^2(\domain)]{\nabla \erpobjhN(u^*)-\nabla \hat{\erpobj}_{N}(u^*)}
			+
			(\alpha^{1/2}/4)\norm[L^2(\domain)]{\quasiinter u^*-\uhN}, \;\;
			\tand 
			\\
			&|\innerp{L^2(\domain)}{\nabla \hat{\erpobj}_N(u^*)-\nabla \erpobj(u^*)}
			{\quasiinter u^* - \uhN}|^{1/2}
			\\& \quad \leq 
			(4/\alpha^{1/2})\norm[L^2(\domain)]{\nabla \hat{\erpobj}_N(u^*)-\nabla 
				\erpobj(u^*)}
			+
			(\alpha^{1/2}/4)\norm[L^2(\domain)]{\quasiinter u^* - \uhN}.
		\end{aligned}
	\end{align*}
	Furthermore, we have
	\begin{align*}
		&|\innerp{L^2(\domain)}{\nabla \erpobjhN(\quasiinter u^*)-
			\nabla \erpobjhN( u^*)}
		{\quasiinter u^*-\uhN}|^{1/2}
		\\& \quad \leq  (4/\alpha^{1/2})
		\norm[L^2(\domain)]{\nabla \erpobjhN(\quasiinter u^*)-
			\nabla \erpobjhN( u^*)}
		+ (\alpha^{1/2}/4)\norm[L^2(\domain)]{\quasiinter u^*-\uhN}.
	\end{align*}
	Combined with \eqref{eq:saa:2020-11-07T23:12:24.196}, 
	and $(\rho_1+\rho_2)^{1/2} \leq \rho_1^{1/2}+ \rho_2^{1/2}$ 
	valid for all $\rho_1$, $\rho_2 \geq 0$,
	we conclude that
	\begin{align}
		\label{eq:reliableestimate_b}
		\begin{aligned}
			(1/4)\norm[L^2(\domain)]{\uhN-\quasiinter u^*}
			& \leq 
			(4/\alpha)\norm[L^2(\domain)]{\nabla \erpobjhN(\quasiinter u^*)-
				\nabla \erpobjhN( u^*)}
			\\
			& \quad +
			(4/\alpha)
			\norm[L^2(\domain)]{\nabla \erpobjhN(u^*)-\nabla \hat{\erpobj}_{N}(u^*)}
			\\
			& \quad +
			(4/\alpha)\norm[L^2(\domain)]{\nabla \hat{\erpobj}_N(u^*)-\nabla 
				\erpobj(u^*)}
			\\
			& \quad +
			(1/\alpha^{1/2})|\innerp{L^2(\domain)}{\nabla \erpobj(u^*)}
			{\quasiinter u^*-u^*}|^{1/2}.
		\end{aligned}
	\end{align}
	Using 	
	\Cref{lem:basicerror_2,ass:2020-01-31T16:12:16.968}, 
	we obtain
	\begin{align}
		\label{eq:reliableestimate_bbbbb}
		\norm[L^2(\domain)]{\nabla \erpobjhN(u^*)-\nabla \erpobjhN(u^*)}
		\leq C_\csp h \big(\alpha + \tfrac{C_\domain^4}{\rdc_{\min}^2}\big)
		\norm[H^1(\domain)]{u^*}.
	\end{align}
	This estimate
	will contribute to the term in \eqref{eq:reliableestimatea}. 
	Using \Cref{lem:basicerror}, we obtain
	the bound 
	\begin{align}
		\label{eq:reliableestimate_bb}
		\norm[L^2(\domain)]{\nabla \erpobjhN(u^*)-\nabla \hat{\erpobj}_{N}(u^*)}
		\leq \oldconstant{gradL2h2} h^2
		\big(\tfrac{C_\domain^2}{\rdc_{\min}}
		\norm[L^2(\domain)]{u^*} + \norm[L^2(\domain)]{y_d}\big)
		= \oldconstant{gradL2h2} h^2 \mathcal{C}^*
	\end{align}
	on the second addend in the right-hand side in 
	\eqref{eq:reliableestimate_b}.
	This error contribution will result in \eqref{eq:reliableestimated}.
	The constant $\mathcal{C}^*$ is defined in \eqref{eq:cstar}.
	
	We must yet derive bounds on the  third and fourth
	term in the right-hand side of \eqref{eq:reliableestimate_b}.
	Since 
	$N \geq 2\ln(2/\delta)/\varepsilon^2$, 
	\Cref{thm:saa:bOgutQ6IqE}
	ensures with a probability of at least $1-\delta$,
	\begin{align}
		\label{eq:reliableestimate_d}
		\norm[L^2(\domain)]{\nabla \hat{\erpobj}_N(u^*)-\nabla \erpobj(u^*)}
		<
		\tfrac{2C_\domain^2}{\rdc_{\min}}
		\big(\tfrac{C_\domain^2}{\rdc_{\min}}
		\norm[L^2(\domain)]{u^*}+\norm[L^2(\domain)]{y_d}\big) \varepsilon
		= 
		\tfrac{2C_\domain^2}{\rdc_{\min}}
		\mathcal{C}^* \varepsilon.
	\end{align}
	This error contribution will result in \eqref{eq:reliableestimateb}.
	\Cref{lem:basicerror_2} yields
	\begin{align}
		\label{eq:reliableestimate_dd}
		|\innerp{L^2(\domain)}{\nabla \erpobj(u^*)}
		{\quasiinter u^*-u^*}|^{1/2}
		\leq C_{\csp}^{1/2} h
		\norm[H^1(\domain)]{\nabla \erpobj(u^*)}^{1/2}
		\norm[H^1(\domain)]{u^*}^{1/2}.
	\end{align}
	Using \Cref{lem:saa:2020-11-08T17:14:08.485,lem:differentiability,lem:saa:H1reg}, we find that
	\begin{align*}
		\begin{aligned}
			\norm[H^1(\domain)]{\nabla \erpobj(u^*)}
			&\leq \alpha \norm[H^1(\domain)]{u^*}
			+ \norm[H^1(\domain)]{\cE{z(u^*, \xi)}}
			\\
			&\leq 
			\alpha \norm[H^1(\domain)]{u^*}
			+ \tfrac{(C_\domain+1)C_\domain}{\rdc_{\min}}
			\big(\tfrac{C_\domain^2}{\rdc_{\min}}
			\norm[L^2(\domain)]{u^*}+\norm[L^2(\domain)]{y_d}\big)
			\\
			& = 
			\alpha \norm[H^1(\domain)]{u^*}
			+ \tfrac{(C_\domain+1)C_\domain}{\rdc_{\min}}
			\mathcal{C}^*.
		\end{aligned}
	\end{align*}
	Hence
	\begin{align}
		\label{eq:reliableestimate_c}
		\begin{aligned}
			\tfrac{1}{\alpha^{1/2}}
			\norm[H^1(\domain)]{\nabla \erpobj(u^*)}^{1/2}
			\norm[H^1(\domain)]{u^*}^{1/2}
			& \leq \norm[H^1(\domain)]{u^*} + 
			(1/\alpha) \norm[H^1(\domain)]{\nabla \erpobj(u^*)} \\
			&\leq 2\norm[H^1(\domain)]{u^*}
			+ \tfrac{(C_\domain+1)C_\domain}{\alpha\rdc_{\min}}
			\mathcal{C}^*.
		\end{aligned}
	\end{align}
	This error contribution will result in \eqref{eq:reliableestimatec}
	and contribute to \eqref{eq:reliableestimatea}.
	Combining 
	\eqref{eq:reliableestimate_a},
	\eqref{eq:reliableestimate_b},
	\eqref{eq:reliableestimate_bbbbb},
	\eqref{eq:reliableestimate_bb},
	\eqref{eq:reliableestimate_d}, 
	\eqref{eq:reliableestimate_dd}, 
	and
	\eqref{eq:reliableestimate_c},
	we obtain
	\eqref{eq:reliableestimate}.
\end{proof}

\section{Numerical illustrations}
\label{sec:simulations}

We illustrate the theoretical results derived in \Cref{prob:femsaa}
(see also \eqref{eq:2021-02-22T15:31:07.382})
empirically on two instances of the optimal control problem
\eqref{eq:saa:lqcp}. Furthermore, we empirically verify
the expectation bound \eqref{eq:expecatation_bound}, which
is implied by \eqref{eq:2021-02-22T15:31:07.382}.
For each instance, 
we present graphical illustrations of an approximate solution
to the control problem \eqref{eq:saa:lqcp} (called ``reference solution'')
and of solutions to a nominal problem and to an ``expected diffusion''
problem associated with \eqref{eq:saa:lqcp}.

Before considering the instances, 
we discuss the setting used for numerical results.
Let $\domain = (0,1)^d$
with $d \in \{1,2\}$. We chose piecewise constant finite elements defined on
a regular triangulation of $[0,1]^d$
with $1/h \in \natural$ 
being the number of cells in each direction to discretize
the control space $L^2(\domain)$  and piecewise linear finite elements
with zero boundary conditions defined on the same triangulation 
to discretize the state space $H_0^1(\domain)$. 
We refer to $h$ as ``mesh width.''
The discretization ensures that the 
discretized feasible set $\adcsph = \adcsp \cap
\csph$ can be 
expressed as
$
\adcsph =  \{\, \sum_{j=1}^{n_h} u_j \phi_h^j \in \csph \colon \,
u_j \in \real, \; \lb \leq u_j \leq \ub, 
\; j = 1, \ldots, n_h\,
\}
$
\cite[Lem.\ 4.4]{Wachsmuth2011}.

To solve the SAA problems, we used \href{http://www.dolfin-adjoint.org/}
{\texttt{dolfin-adjoint}} \cite{Mitusch2019,Funke2013}
with \texttt{FEniCs} \cite{Alnes2015,Logg2012} 
and a semismooth Newton-CG method \cite{Mannel2020,Stadler2009,Ulbrich2011}
applied to a normal map reformulation of the
first-order optimality conditions
for \eqref{eq:saa:femsaa}. The implementation adapts
that of \href{https://github.com/funsim/moola}{\texttt{Moola}}'s 
\texttt{NewtonCG} \cite{Nordaas2016}. We chose
the zero control as an initial point for each SAA problem.

To empirically verify the expectation bound in \eqref{eq:expecatation_bound},
we approximate the expectation $\cE{\norm[L^2(\domain)]{\uhN-u^*}}$ using 
$48$ samples.  To demonstrate the reliable error estimate, we exploit
the fact that it implies a bound on the Luxemburg norm
of $\uhN-u^*$. Let $\hsp$ be a real, separable Hilbert space.
We define the Luxemburg norm $\luxemburg[\hsp]{2}{\cdot}$ 
of a random vector $Z : \Omega \to \hsp$ by
\begin{align}
	\label{eq:mlmc:2020-01-13T12:07:32.995}
	\luxemburg[\hsp]{2}{Z}
	= \inf_{\tau > 0} \, \{ \, \tau : \,
	\cE{\youngfun{\norm[\hsp]{Z}/\tau}} 
	\leq 1 \, \}
	= \inf_{\tau > 0} \, \{ \, \tau : \,
	\cE{\eu^{\norm[\hsp]{Z}^2/\tau^2}} 
	\leq 2 \, \},
\end{align} 
with the Young function $\youngfun{x}= \eu^{x^2}-1$. 
The Orlicz space 
$L_{\youngfun{\cdot}}(\Omega; \hsp) = 
L_{\youngfun{\norm[\hsp]{\cdot}}}(\Omega; \hsp)$
consists of all random vectors mapping from $\Omega$ to $\hsp$
with finite Luxemburg norm; see  \cite[Chap.\ 6]{Kosmol2011}.
The exponential tail bound 
\eqref{eq:2021-02-22T15:31:07.382} implies
\begin{align}
	\label{eq:2021-02-15T18:52:21.461}
	\luxemburg[L^2(\domain)]{2}{\uhN-u^*} 
	\leq
	3\sqrt{2}(c_1 h  + (c_2/\sqrt{N}));
\end{align}
see \cref{sec:tailluxemburg}. 
We approximate the expectation in \eqref{eq:mlmc:2020-01-13T12:07:32.995}
using the same samples used to estimate $\cE{\norm[L^2(\domain)]{\uhN-u^*}}$.
The simulations used to compute the SAA solutions
were performed on  PACE's Phoenix cluster \cite{PACE} with
Dual Intel Xeon Gold 6226 2.7 GHz CPUs.

Besides those for reference solutions,
we provide visualizations
of the solution to the nominal problem
\begin{align}
	\label{eq:nominalproblem}
	\min_{u_h\in\adcsph}\, 
	(1/2)\norm[L^2(\domain)]{S_{h}(u_h,\cE{\xi})-y_d}^2
	+(\alpha/2)\norm[L^2(\domain)]{u_h}^2
	+ \gamma \norm[L^1(\domain)]{u_h}.
\end{align}
and for the solution to the ``expected diffusion'' problem
\begin{align}
	\label{eq:avgproblem}
	\min_{u_h\in\adcsph}\, 
	(1/2)\norm[L^2(\domain)]{G_h(u_h)-y_d}^2
	+(\alpha/2)\norm[L^2(\domain)]{u_h}^2
	+ \gamma \norm[L^1(\domain)]{u_h},
\end{align}
where  for each $u \in L^2(\domain)$,  
$w_h = G_h(u) \in \ssph$ solves 
\begin{align*}
	\int_{\domain} \cE{\rdc(\xi)}
	\nabla w_{h} \cdot \nabla v_h \, \du x = 
	\int_{\domain} uv_h  \, \du x
	\quad \tfa \quad v_h \in \ssph.
\end{align*}
For the random diffusion coefficients
considered in \cref{subsect:1d,subsect:2d}, the expected value
$\cE{\kappa(\xi)(x)}$ can be computed explicitly 
for each $x \in \bar{\domain}$. We use this fact
for the numerical solution of \eqref{eq:avgproblem}.
To graphically visualize controls, we interpolate the piecewise constant 
controls to the discretized state space.

\subsection{One-dimensional equation}
\label{subsect:1d}
We consider an instance of the control
problem \eqref{eq:saa:lqcp} with $d=1$,
$\alpha =  0{.}001$, 
$\gamma = 0{.}01$,
$\lb = -3$,
$\ub = 3$,
$y_d(x) = -1/2$ if $x \in  (1/4, 3/4)$
and $y_d(x) = 1$ otherwise.
We chose $\Xi = [-1,1]^4$ and 
\begin{align}
	\label{eq:kappa:1d}
	\kappa_{\sigma}(\xi)(x) = 
	\exp\big(
	\sigma\xi_1 \cos(1.1\pi x)
	+\sigma\xi_2 \cos(1.2\pi x)
	+\sigma\xi_3 \sin(1.3\pi x)
	+\sigma\xi_4 \sin(1.4\pi x)
	\big),
\end{align}
where $\sigma > 0$ is a parameter;
cf.\ \cite[eq.\ (8.1)]{Martin2021}.
The random variables $\xi_1, \ldots, \xi_4$ are independent,
each with $[-1,1]$-uniform distribution.
The probability distribution of $\xi$ was approximated 
by a discrete uniform distribution
supported on the grid points of a uniform mesh of $\Xi$ using 
$32$ grid points
in each direction, yielding a discrete distribution with 
$32^4=2^{20}$ 
scenarios. Samples for the SAA problems were generated from this
discrete distribution. The reference solution $u^* = u_{h^*,N^*}$
to \eqref{eq:saa:lqcp} was computed with  $h^* = 2^{-10}$ and
$N^* = (1/h^*)^2 = 2^{20}$.

The reference solutions
for $\sigma \in \{1,2\}$ are depicted in \Cref{subfig:refsol:oned}.
The solution to the nominal problem
\eqref{eq:nominalproblem} and those to 
the ``expected diffusion'' problems with $h = h^*$
and $\sigma \in \{1,2\}$ are depicted in 
\Cref{subfig:nomavg:oned}. 
The convergence rates depicted in \Cref{subfig:hN:oned} 
for $\sigma=1$ are close
to the theoretical rates provided by
\eqref{eq:expecatation_bound} and \eqref{eq:2021-02-15T18:52:21.461}.
For $\sigma \in \{1,2\}$, 
\Cref{subfig:errorcontribution:sigma1:oned,%
	subfig:errorcontribution:sigma2:oned} 
depict $48$ samples, and the empirical mean and Luxemburg norm of
\begin{align}
	\label{eq:gradientestimate}
	\norm[L^2(\domain)]{\nabla \erpobjhN(\quasiinter u^*)- \nabla \hat{\erpobj}_{h^*,N^*}(u^*)},
\end{align}
which approximates 
$\norm[L^2(\domain)]{\nabla \erpobjhN(\quasiinter u^*)- \nabla \erpobj( u^*)}$ in \eqref{eq:simpleerrorbound}. 
The convergence rates depicted in \Cref{subfig:errorcontribution:sigma1:oned,%
	subfig:errorcontribution:sigma2:oned}
are close to the expected rates.
Since the term in \eqref{eq:gradientestimate} does not depend
on the SAA solutions $\uhN$, it is computationally cheaper to
evaluate than $\norm[L^2(\domain)]{\uhN-u^*}$.

\begin{figure}[t]
	\centering
	\subfloat[Reference solutions.]{%
		\includegraphics[width=0.33\textwidth]
		{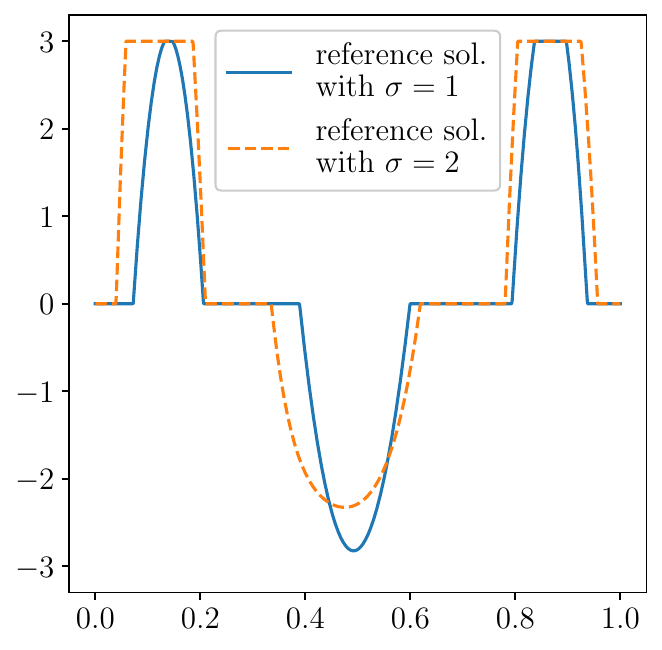}
		\label{subfig:refsol:oned}}
	\hfil
	\subfloat[Nominal and ``expected diffusion'' solutions.]{%
		\includegraphics[width=0.33\textwidth]
		{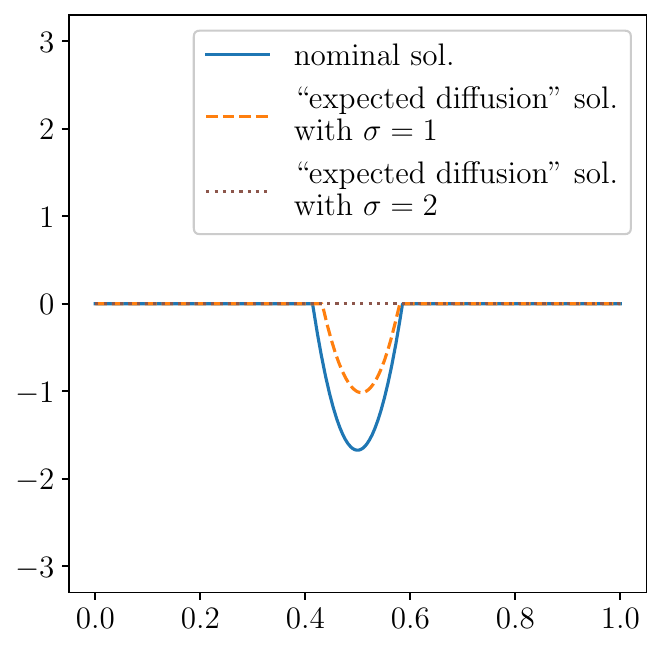}
		\label{subfig:nomavg:oned}}
	\caption{For the problem considered in \cref{subsect:1d},
		reference solutions $u^* = u^*_{h^*,N^*}$
		with $h^* = 2^{-10}$ and $N^* = 2^{20}$ \emph{(a)}, 
		solution to the nominal problem \eqref{eq:nominalproblem} 
		and solutions to the ``expected diffusion'' 
		problems \eqref{eq:avgproblem} with $h = h^*$ \emph{(b)}.
		The parameter $\sigma$ refers to that used in 
		the random field \eqref{eq:kappa:1d}.
		The subfigures have the same axes limits.}
\end{figure}

\begin{figure}[t]
	\centering
	\subfloat[Control errors for $\sigma=1$.]{%
		\includegraphics[width=0.325\textwidth]
		{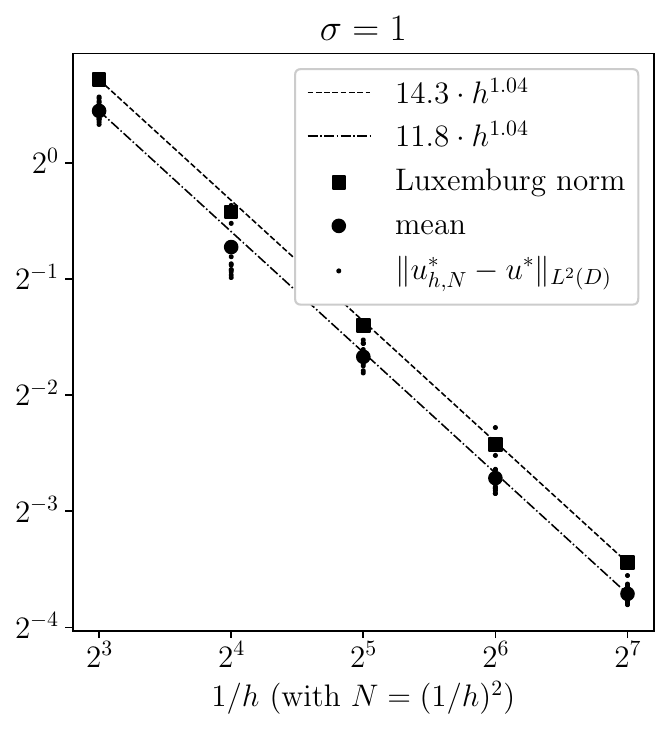}
		\label{subfig:hN:oned}}
	\hfil 
	\subfloat[Gradient errors for $\sigma=1$.]{%
		\includegraphics[width=0.325\textwidth]
		{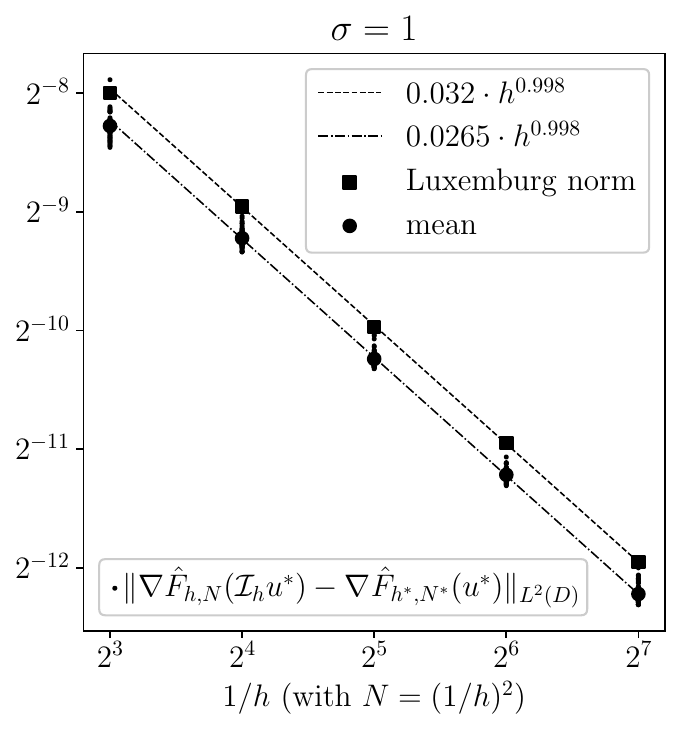}
		\label{subfig:errorcontribution:sigma1:oned}}
	\hfil 
	\subfloat[Gradient errors for $\sigma=2$.]{%
		\includegraphics[width=0.325\textwidth]
		{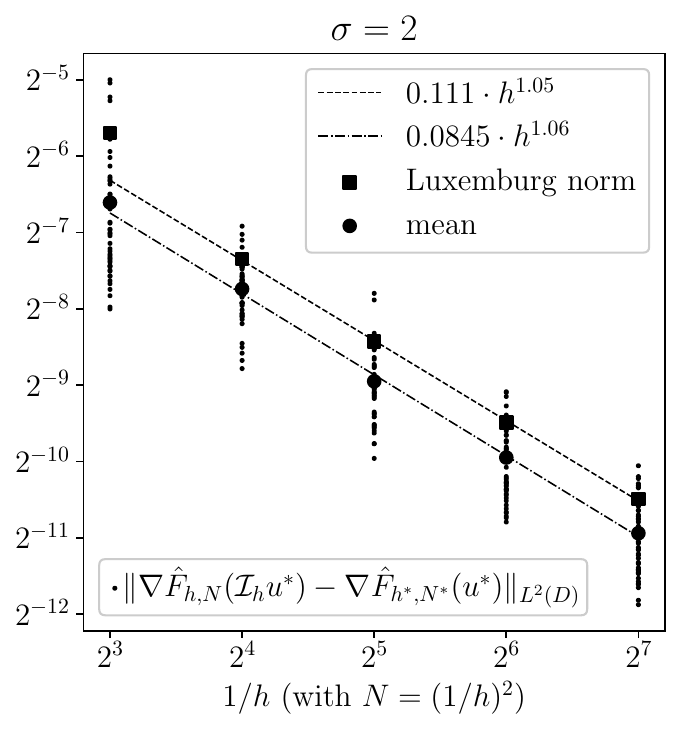}
		\label{subfig:errorcontribution:sigma2:oned}}
	\caption{For the problem considered in \cref{subsect:1d},
		realizations of the error $\norm[L^2(\domain)]{\uhN-u^*}$
		and their empirical  mean and Luxemburg norm
		as a function of the inverse mesh width $1/h$
		with $N  = (1/h)^2$ and $\sigma=1$ \emph{(a)}, 
		and realizations of 
		$\norm[L^2(\domain)]{\nabla \erpobjhN(\quasiinter u^*)- \nabla \hat{\erpobj}_{h^*,N^*}(u^*)}$
		and their empirical  mean and Luxemburg norm
		as a function of the inverse mesh width $1/h$
		with $N  = (1/h)^2$ and $\sigma=1$ \emph{(b)}
		and $\sigma=2$ \emph{(c)}.
		Here $u^* = u^*_{h^*,N^*}$
		with $h^* = 2^{-10}$ and $N^* = 2^{20}$
		is the reference solution, and
		$\nabla \hat{\erpobj}_{h^*,N^*}(u^*)$ is the reference
		gradient.
		For subfigures \textnormal{(b)} and \textnormal{(c)}, the
		first empirical estimates were excluded from the computation
		of the least squares fits. 
		The parameter $\sigma$ refers to that used in 
		the random field \eqref{eq:kappa:1d}.} 
\end{figure}

\subsection{Two-dimensional equation}
\label{subsect:2d}
We consider an instance of problem \eqref{eq:saa:lqcp} with $d=2$,
$\alpha =  0{.}001$, 
$\gamma = 0{.}01$,
$\lb = -6$,
$\ub = 6$,
$y_d(x) = -1$ if $x \in  (1/4, 3/4)^2$
and $y_d(x) = 1$ otherwise
(cf.\ \cite[p.\ 49]{Vallejos2008}).
We chose $\Xi = [-1,1]^4$ and 
the random diffusion coefficient
(cf.\ \cite[eq.\ (8.1)]{Martin2021})
\begin{align*}
	\kappa(\xi)(x) = 
	\exp\big(\xi_1 \cos(1.1\pi x_1)
	+\xi_2 \cos(1.2\pi x_1)
	+\xi_3 \sin(1.3\pi x_2)
	+\xi_4 \sin(1.4\pi x_2)\big).
\end{align*}
The random variables $\xi_1, \ldots, \xi_4$ are independent,
each with $[-1,1]$-uniform distribution.
The small number of random variables allows us to 
obtain an accurate reference solution to \eqref{eq:saa:lqcp}.
For our setting,
\Cref{ass:domainkappa,ass:2020-11-08T15:31:39.527,%
	ass:csph,ass:2020-01-31T16:12:16.968} hold true.
The probability distribution of $\xi$ was approximated 
by a discrete uniform distribution
supported on the grid points of a uniform mesh of $\Xi$ using $12$ grid points
in each direction, yielding a discrete distribution with 
$144^2=20736$ scenarios. Samples for the SAA problems were generated from this
discrete distribution. The reference solution $u^* = u_{h^*, N^*}$
to \eqref{eq:saa:lqcp} was computed with  $h^* = 1/144$ 
and $N^* = 144^2$.

\Cref{fig:solutions:twod} depicts the reference solution,
the solution to the nominal problem
\eqref{eq:nominalproblem} and that to 
the ``expected diffusion'' problem \eqref{eq:avgproblem} 
with $h = h^*$. 		
The controls depicted in \Cref{subfig:nomsol,subfig:avgsol}
differ from the reference solution depicted in
\Cref{subfig:refsol}.

\Cref{subfig:hN} illustrates the theoretical bounds
\eqref{eq:expecatation_bound} and \eqref{eq:2021-02-15T18:52:21.461}
empirically. 
The convergence rates shown in \Cref{fig:hN}
were computed using least squares. 
We used the mesh widths $h \in \{1/8, 1/12, 1/16, 1/24, 1/36, 1/48, 1/72\}$
with corresponding sample sizes $N \in \{8^2, 12^2,\ldots, 72^2\}$
to perform the simulations depicted in 
\Cref{subfig:hN}. While these mesh widths are not equidistant
on a binary logarithmic scale, they are multiplies of $1/144$, which facilitates
the computation of the errors $\norm[L^2(\domain)]{\uhN-u^*}$.
The convergence rates depicted in \Cref{subfig:hN} are close to the
theoretical rates. For a fixed mesh width $h$, 
\Cref{subfig:h} depicts the error's empirical Luxemburg norm
and mean as a function of the sample size $N$. Moreover, 
for a fixed  sample size $N$,
\Cref{subfig:N} shows the  error's Luxemburg norm
and mean  as a function of the inverse mesh width $1/h$.
The simulation output depicted in \Cref{fig:hN} may suggest that
small errors cannot be obtained by keeping either the mesh width
or the sample size fixed.

\begin{figure}[t]
	\centering
	\subfloat[Reference solution.]{%
		\includegraphics[width=0.325\textwidth]
		{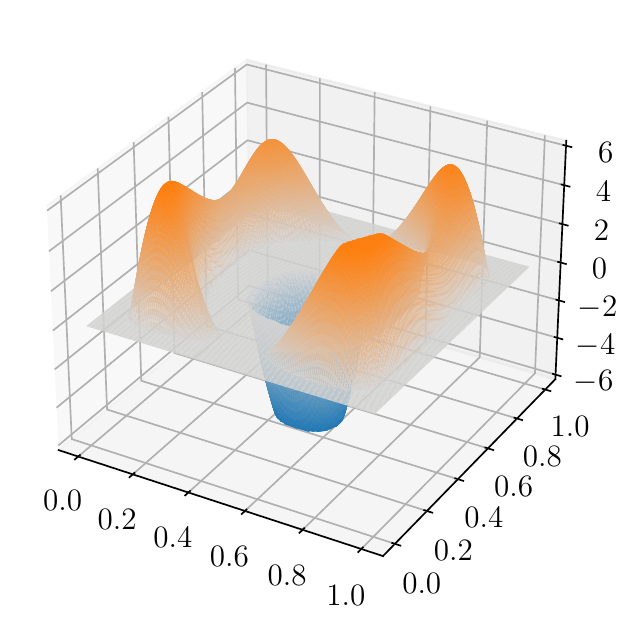}
		\label{subfig:refsol}}
	\hfil 
	\subfloat[Nominal solution.]{%
		\includegraphics[width=0.325\textwidth]
		{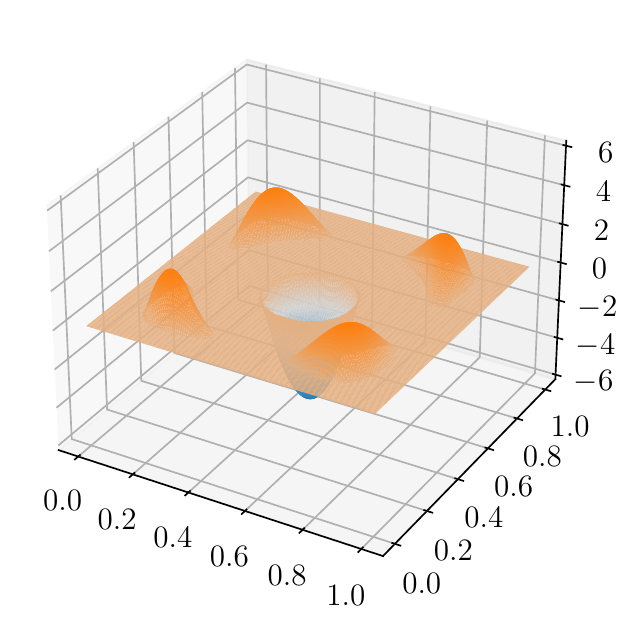}
		\label{subfig:nomsol}}
	\hfil 
	\subfloat[``Expected diffusion'' solution.]{%
		\includegraphics[width=0.325\textwidth]
		{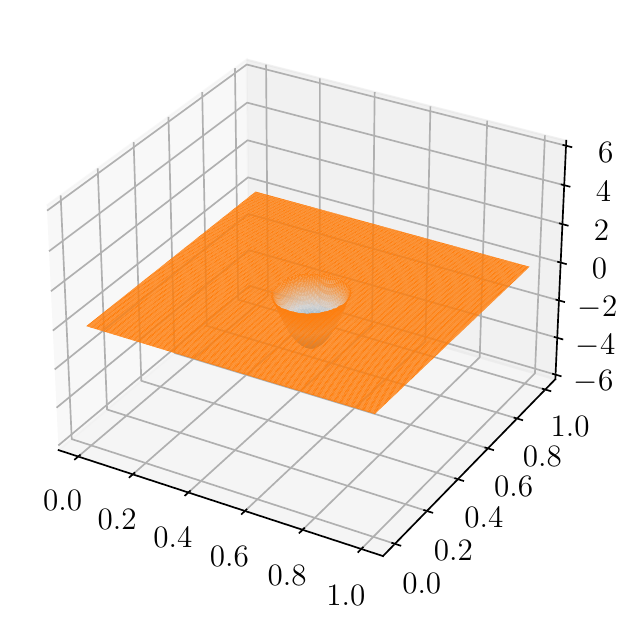}
		\label{subfig:avgsol}}
	\caption{For the problem  considered in \cref{subsect:2d},
		the reference solution $u^* = u^*_{h^*,N^*}$
		with $h^* = 1/144$ and $N^* = 144^2$
		\textnormal{(a)},
		solution to the nominal problem \eqref{eq:nominalproblem}
		with $h = h^*$
		\textnormal{(b)},
		and solution to 
		the ``expected diffusion'' problem \eqref{eq:avgproblem}
		with $h = h^*$
		\textnormal{(c)}. The subfigures have the same axes limits.}
	\label{fig:solutions:twod}
\end{figure}

\begin{figure}[t]
	\centering
	\subfloat[Variable $h$ with $N=(1/h)^2$.]{%
		\includegraphics[width=0.33\textwidth]
		{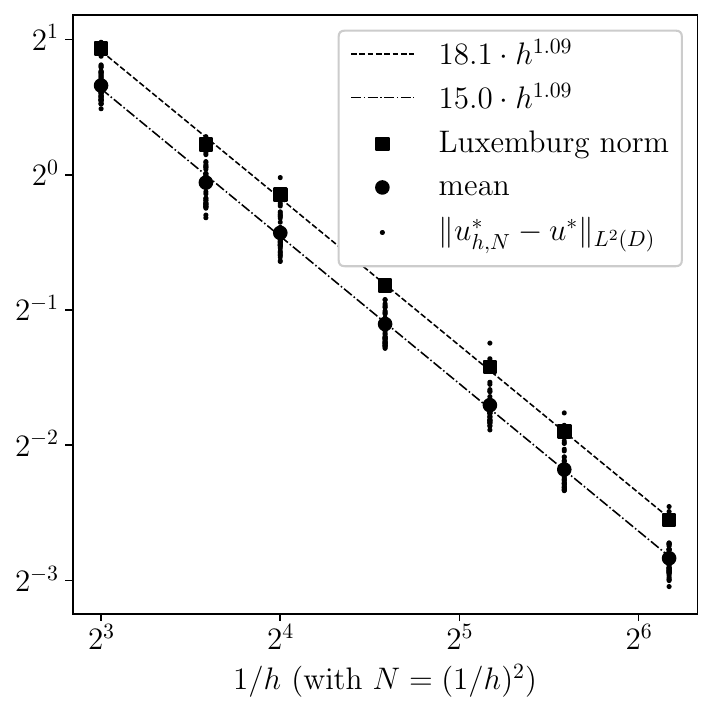}
		\label{subfig:hN}}
	\hfil 
	\subfloat[Fixed mesh width $h$.]{%
		\includegraphics[width=0.33\textwidth]
		{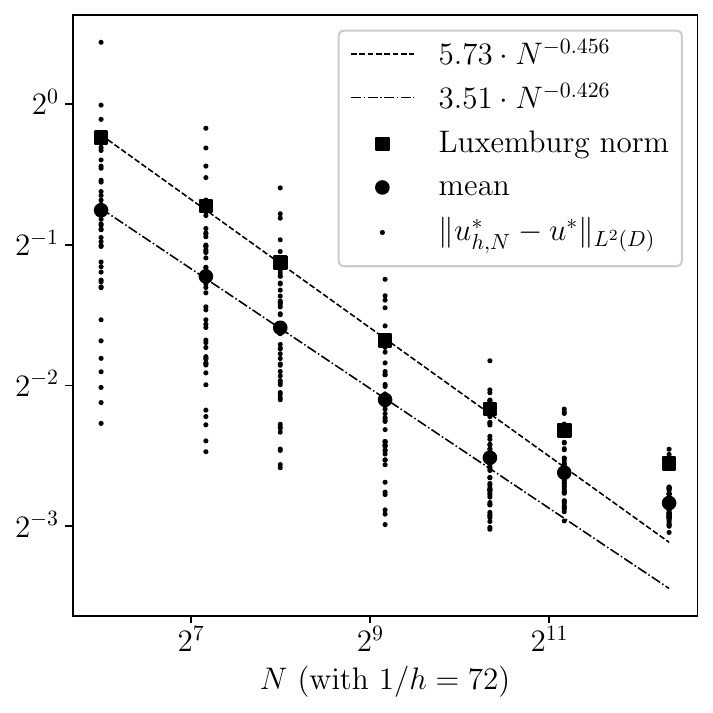}\label{subfig:h}}
	\hfil 
	\subfloat[Fixed sample size $N$.]{%
		\includegraphics[width=0.33\textwidth]
		{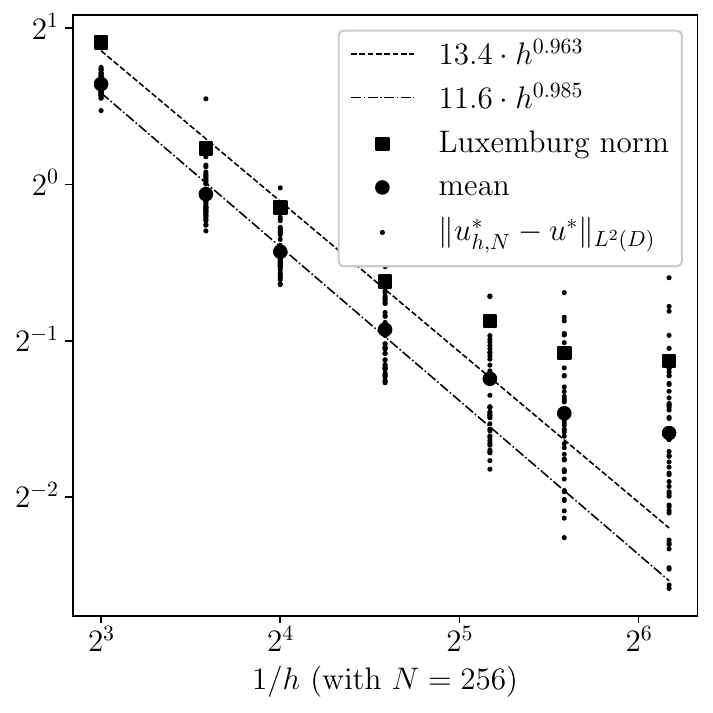}\label{subfig:N}}
	\caption{For the problem  considered in \cref{subsect:2d},
		realizations of the error $\norm[L^2(\domain)]{\uhN-u^*}$
		and their empirical  mean and Luxemburg norm as a function
		of the inverse mesh width $1/h$
		with sample size $N  = (1/h)^2$ \textnormal{(a)}, 
		of the sample size for fixed
		mesh width \textnormal{(b)}, 
		and of the inverse mesh width for fixed sample
		size \textnormal{(c)}.
		Here $u^* = u_{h^*,N^*}^*$ is the reference solution
		with $h^* = 1/144$ and $N^* = 144^2$. 
		For subfigures \textnormal{(b)} and \textnormal{(c)}, the 
		last three empirical estimates were excluded
		from the computation of the least squares fit.}
	\label{fig:hN}
\end{figure}

\section{Discussion}
\label{sec:discussion}

The numerical solution of risk-neutral optimal control problems
governed by PDEs with random inputs require both approximations of expected
values and space discretizations.
We have applied the SAA method and a Galerkin-type approximation
to a risk-neutral 
linear elliptic PDE-constrained optimization problem and derived
an exponential tail bound for the distance between the solution to 
the finite dimensional SAA problem and that to the risk-neutral program. 
The exponential tail bound implies that SAA solutions are close to the
risk-neutral problem's solution with high probability. The bound can be used
to balance errors caused
by space discretizations and Monte Carlo sample-based approximations.
Our main tool for deriving the exponential tail bound is a large deviation-type
bound derived in \cite{Pinelis1992,Pinelis1994}.
We performed numerical simulations illustrating our theoretical error bound.

We have derived reliable  error estimates
for a basic model problem with $L^1(\domain)$-norm 
control regularization.
Our derivation of the reliable error estimate
\eqref{eq:2021-02-22T15:31:07.382} relies on \Cref{ass:domainkappa_2},
which is violated for lognormal random fields. For a lognormal fields, 
we conjecture that an exponential tail bound similar to that 
in \eqref{eq:2021-02-22T15:31:07.382} does not hold. 
However, even for lognormal  fields, 
the estimate \eqref{eq:saa:2020-11-07T23:12:24.196} can be used
to establish bounds on $\cE{\norm[L^2(\domain)]{\uhN-u^*}^2}$
in terms of the sample size $N$ and discretization parameter $h$.
Our results may be extended in several ways. 
For example, rather than discretizing the control space using finite element
functions, the variational discretization approach
\cite{Hinze2005a} can be considered. 
In order to numerically solve the discretized 
SAA problems, we implicitly assumed that 
bilinear forms in
the discretized PDEs can be computed exactly. However, 
without structural assumptions on $\rdc$, 
the random diffusion coefficient $\rdc$ must be
approximated using, for example, projection or interpolation.
Furthermore, our error analysis exploits
that realizations of the random diffusion coefficient are continuously
differentiable. The error analysis may be extended to allow for 
rougher random diffusion coefficients, but the convergence
with respect to the space discretization may then be slower.  Moreover, 
our error analysis relies on the linearity
of the  elliptic PDE with respect to the control variable.
Since many PDE-constrained optimization problems under uncertainty
are governed by nonlinear PDEs, it would be desirable to 
have available reliable
error estimates for nonlinear optimal control problems
under uncertainty.

\appendix

\section{Tail bound implies expectation bound}
\label{sec:expectation_bound}
We show that
the exponential tail bound
\eqref{eq:2021-02-22T15:31:07.382} implies 
the expectation bound \eqref{eq:expecatation_bound}.
Using \cite[Lem.\ 3.4]{Kallenberg2002}, we have
$
\cE{\norm[L^2(\domain)]{\uhN-u^*}}
= \int_{0}^{\infty}\Prob{\norm[L^2(\domain)]{\uhN-u^*} \geq r} \du r
$.
Combined with \eqref{eq:2021-02-22T15:31:07.382},
$h$, $c_1$, $c_2 > 0$
and 
$2\int_{0}^\infty \exp(-x^2/(2\sigma^2)) \du 
x = \sqrt{2\pi \sigma^2}$
with $\sigma^2 = c_2^2/N > 0$, we find that
\begin{align*}
	\cE{\norm[L^2(\domain)]{\uhN-u^*}}
	&\leq 
	c_1h 
	+\int_{0}^{\infty}\Prob{\norm[L^2(\domain)]{\uhN-u^*} \geq c_1h+
		r} \du r
	\\
	& \leq  
	c_1h 
	+ \int_{0}^\infty 2\exp(-r^2N/(2c_2^2))  \du r
	=
	c_1 h + c_2\sqrt{2\pi}/\sqrt{N}.
\end{align*}
Hence the exponential tail bound
\eqref{eq:2021-02-22T15:31:07.382} implies 
the expectation bound \eqref{eq:expecatation_bound}.
\section{Tail bound implies bound on Luxemburg norm}
\label{sec:tailluxemburg}
We show that the exponential tail bound \eqref{eq:2021-02-22T15:31:07.382}
implies the Luxemburg norm bound
provided in  \eqref{eq:2021-02-15T18:52:21.461}
and as a consequence we obtain bounds on all finite moments
of $\norm[L^2(\domain)]{\uhN-u^*}$.
Rewriting \eqref{eq:2021-02-22T15:31:07.382}
and defining $\rv = \norm[L^2(\domain)]{\uhN-u^*} $,  we have
\begin{align}
	\label{eq:2021-02-22T15:31:07.382_2}
	\Prob{|\rv|\geq c_1 h + (c_2/
		\sqrt{N}) \varepsilon} \leq 
	2\exp(-\varepsilon^2/2) 
	\quad \tfa \quad \varepsilon > 0.
\end{align}
We have $h$, $c_1$, $c_2 > 0$.
If $\varepsilon \geq 1$, 
then 
$|\rv| \geq c_1 h \varepsilon + (c_2/
\sqrt{N}) \varepsilon$
yields
$|\rv| \geq c_1 h  + (c_2/
\sqrt{N}) \varepsilon$.
We also have  $2\exp(-\varepsilon^2/2) > 1$
for all $\varepsilon \in [0, 1]$. Hence
\eqref{eq:2021-02-22T15:31:07.382_2} implies
\begin{align*}
	\Prob[\big]{|\rv| \geq 
		\big(c_1 h  + (c_2/
		\sqrt{N})\big) \varepsilon} \leq 
	2\exp(-\varepsilon^2/2) 
	\quad \tfa \quad \varepsilon > 0.
\end{align*}
Combined with
\cite[Thm.\ 3.4 on p.\ 56]{Buldygin2000}
(used with $\varphi(x) = x^2$ \cite[pp.\ 42 and 55]{Buldygin2000},
$C = 2$ and $D = \sqrt{2}\big(c_1 h  + (c_2/
\sqrt{N})\big)$) and eq.\ \eqref{eq:mlmc:2020-01-13T12:07:32.995}, 
we obtain the bound
\begin{align*}
	\luxemburg[L^2(\domain)]{2}{\uhN-u^*}
	= \luxemburg[\real]{2}{\rv}
	\leq  3\sqrt{2}(c_1 h  + (c_2/\sqrt{N})).
\end{align*}
Hence the exponential tail bound \eqref{eq:2021-02-22T15:31:07.382}
yields the Luxemburg norm bound in \eqref{eq:2021-02-15T18:52:21.461}.
Using \cite[Lem.\ 3.4 on p.\ 58]{Buldygin2000}, we further have
$\cE{\exp(\lambda |Z|)} \leq 2 \exp(\lambda^2 \varrho^2/4)$
for all $\lambda \in \real$, where $\varrho = 3\sqrt{2}\big(c_1 h  + (c_2/
\sqrt{N})\big)$. Combined with the computations in 
\cite[p.\ 7]{Buldygin2000} (used with $\tau=\varrho/\sqrt{2}$), 
\begin{align}
	\label{eq:new_expecatation_bound}
	\cE{\norm[L^2(\domain)]{\uhN-u^*}^p}
	\leq 3^p2(p/\mathrm{e})^{p/2}
	\big(c_1 h  + (c_2/\sqrt{N})\big)^p
	\quad \tfa \quad p > 0.
\end{align}
\section{Higher regularity of a proximity operator}
We establish a higher regularity result of a certain proximity 
operator, which is essentially known 
\cite{Borzi2005,Casas2012a,Falk1973,Geveci1979,%
	Malanowski1982,Troeltzsch2010a,Wachsmuth2011}.
It is used in \Cref{lem:saa:H1reg} to establish
higher regularity of the solution to \eqref{eq:saa:lqcp}.
For $\wlb$, $\wub \in L^2(\domain)$, we define $[\wlb,\wub] 
= \{\, u \in L^2(\domain) :\, \wlb \leq u \leq \wub  \,\, 
\text{a.e. in} \,\, \domain \,\}$.
\begin{lemma}
	\label{lem:Feb0420211008}
	Let $\domain \subset \real^d$
	be a bounded Lipschitz domain, and let 
	$\wlb$, $\wub \in H^1(\domain) \cap L^\infty(\domain)$
	fulfill $\wlb \leq \wub$ a.e.\ in $\domain$.
	We define $\psi : L^2(\domain) \to [0,\infty]$
	by 	$\psi(v) = \mu \norm[L^1(\domain)]{v} + I_{[\wlb,\wub]}(v)$
	with $\mu \in [0,\infty)$.
	Let $\prox{\psi}{} : L^2(\domain) \to L^2(\domain)$
	be the proximity operator of $\psi$ as defined in \eqref{eq:prox}. 
	Then $\prox{\psi}{v} \in H^1(\domain)$
	for all $	v \in H^1(\domain)$ and
	$$\norm[H^1(\domain)]{\prox{\psi}{v}} 
	\leq 
	\norm[H^1(\domain)]{\wlb} 
	+ \norm[H^1(\domain)]{\wub}
	+ \eqeqnorm[H^1(\domain)]{v}
	\quad \tfa \quad 
	v \in H^1(\domain).
	$$
\end{lemma}
We prove \Cref{lem:Feb0420211008} using \Cref{lem:Feb0420211008_2}.
If $\wlb$, $\wub \in L^2(\domain)$ with $a\leq b$ a.e.\ in $\domain$, then
$\prox{I_{[\wlb,\wub]}}{}$ equals the
projection operator onto $[\wlb,\wub]$.
\begin{lemma}
	\label{lem:Feb0420211008_2}
	Let hypotheses of \Cref{lem:Feb0420211008} hold, 
	and let $v$, $w \in H^1(\domain)$. Then
	\begin{enumerate}[nosep,wide,leftmargin=*,before={\parskip=0pt}]
		\item 
		$\prox{I_{[\wlb,\wub]}}{w} \in H^1(\domain)$
		and
		$
		\eqeqnorm[H^1(\domain)]{\prox{I_{[\wlb,\wub]}}{w}}^2 
		\leq 
		\eqeqnorm[H^1(\domain)]{\wlb}^2
		+ \eqeqnorm[H^1(\domain)]{\wub}^2
		+ \eqeqnorm[H^1(\domain)]{w}^2
		$.
		\item 
		$v-\prox{I_{[-\mu,\mu]}}{v} \in H^1(\domain)$
		and
		$\eqeqnorm[H^1(\domain)]{v-\prox{I_{[-\mu,\mu]}}{v}}
		\leq \eqeqnorm[H^1(\domain)]{v}$.
	\end{enumerate}
\end{lemma}
\begin{proof}
	In this
	proof, we omit writing evaluations at $x \in \domain$.
	
	\textnormal{(a)} We have $\prox{I_{[\wlb,\wub]}}{w} \in H^1(\domain)$
	\cite[pp.\ 114--115]{Troeltzsch2010a}.
	Since
	$\prox{I_{[\wlb,\wub]}}{w}  = \wlb$
	if $w < \wlb$, 
	$\prox{I_{[\wlb,\wub]}}{w}  = w$
	if $w \in [\wlb, \wub]$, and
	$\prox{I_{[\wlb,\wub]}}{w}  = \wub$ otherwise,
	we obtain the estimate.
	
	\textnormal{(b)}
	We have
	$\prox{I_{[-\mu,\mu]}}{v} \in H^1(\domain)$
	\cite[pp.\ 114--115]{Troeltzsch2010a}
	and hence
	$v-\prox{I_{[-\mu,\mu]}}{v} \in H^1(\domain)$.
	Combined with
	$v-\prox{I_{[-\mu,\mu]}}{v} = v+\mu$
	if $v < -\mu$, 
	$v-\prox{I_{[-\mu,\mu]}}{v} = 0$
	if $v \in [-\mu,\mu]$, and
	$v-\prox{I_{[-\mu,\mu]}}{v} = v-\mu$
	otherwise, we obtain the  estimate.
\end{proof}

\begin{proof}[{Proof of \Cref{lem:Feb0420211008}}]
	The set $[\wlb,\wub] \subset L^2(\domain)$ is nonempty, convex and closed
	\cite[pp.\ 116--117]{Troeltzsch2010a}.
	Hence the indicator function $I_{[\wlb,\wub]}$ is proper, convex and 
	lower semicontinuous \cite[Ex.\ 2.115]{Bonnans2013}.
	Combined with the boundedness of $\domain$
	and H\"older's inequality, we find that $v \in L^1(\domain)$
	and
	$\norm[L^1(\domain)]{v} \leq 
	\norm[L^1(\domain)]{1}^{1/2} \norm[L^2(\domain)]{v} < \infty$
	for each $v \in L^2(\domain)$.
	Hence $\psi$ is  proper, convex and 
	lower semicontinuous.
	Fix $v \in H^1(\domain)$. 
	Since $L^2(\domain)$ is decomposable
	\cite[p.\ 677]{Rockafellar2009}
	and $\prox{\psi}{v}$ is well-defined \cite[p.\ 211]{Bauschke2011},
	the theorem
	on the interchange of minimization and integration
	\cite[Thm.\ 14.60]{Rockafellar2009} ensures
	for almost every $x \in \domain$,
	\begin{align*}
		\prox{\psi}{v}(x) = 
		\prox{\mu |\cdot| + I_{[\wlb(x), \wub(x)]}}{v(x)}.
	\end{align*}
	Similarly, for each $w \in L^2(\domain)$, we obtain
	for almost every $x \in \domain$,
	\begin{align*}
		\prox{I_{[-\mu,\mu]}}{w}(x)
		= \prox{I_{[-\mu,\mu]}}{w(x)}
		\;\; \tand \;\;
		\prox{I_{[\wlb(x),\wub(x)]}}{w(x)}
		= \prox{I_{[\wlb(x),\wub(x)]}}{w}(x).
	\end{align*}
	Since 
	$\prox{\mu |\cdot| + I_{[t_1, t_2]}}{t_3}
	= \prox{I_{[t_1,t_2]}}{t_3-\prox{I_{[-\mu,\mu]}}{t_3}}
	$
	for all $t \in \real^3$ \cite[Ex.\ 3.2.9]{Milzarek2016},
	we have
	\begin{align}
		\label{eq:Feb0420210911}
		\prox{\psi}{v} = 
		\prox{I_{[\wlb,\wub]}}{v-\prox{I_{[-\mu,\mu]}}{v}}.
	\end{align}
	Combined with \Cref{lem:Feb0420211008_2}, it follows that
	$\prox{\psi}{v} \in H^1(\domain)$ and
	\begin{align}
		\label{eq:proxH1}
		\eqeqnorm[H^1(\domain)]{\prox{\psi}{v}}^2 
		\leq 
		\eqeqnorm[H^1(\domain)]{\wlb}^2
		+ \eqeqnorm[H^1(\domain)]{\wub}^2
		+ \eqeqnorm[H^1(\domain)]{v}^2.
	\end{align}
	Using \eqref{eq:Feb0420210911}, we have $\prox{\psi}{v} \in [\wlb,\wub]$.
	Since $\wlb(x) \leq \prox{\psi}{v}(x)\leq \wub(x)$ ensures 
	the estimate
	$|\prox{\psi}{v}(x)|^2 \le \wub(x)^2 + \wlb(x)^2$, we obtain
	$$
	\norm[L^2(\domain)]{\prox{\psi}{v}}^2 \leq 
	\norm[L^2(\domain)]{\wlb}^2 + \norm[L^2(\domain)]{\wub}^2.
	$$
	Combined with \eqref{eq:proxH1} and 
	$(\rho_1+\rho_2)^{1/2} \leq \rho_1^{1/2}+\rho_2^{1/2}$
	valid for all $\rho_1$, $\rho_2 \geq 0$, we obtain
	the stability estimate.
\end{proof}

\section*{Acknowledgments}
JM thanks the two anonymous referees and the associate editor
for their helpful comments and suggestions. 
JM thanks Dominik Hafemeyer for our discussion on empirically verifying
convergence rates and Gernot Holler  for encouraging the author to
make explicit the dependence on problem-dependent parameters in the reliable
error estimate.

\begin{footnotesize}
	\bibliography{femsaa_revised.bbl}
\end{footnotesize}

\end{document}